\newcommand{\doublespace} {\addtolength{\baselineskip}{.50\baselineskip}}

\documentclass[reqno,10pt]{amsart}
\usepackage {amsmath, amssymb, enumerate, amsthm}
\usepackage{graphicx}
\usepackage{cite}

\newtheorem{theorem}{Theorem}[section]
\newtheorem{lemma}{Lemma}[section]

\newtheorem{proposition}{Proposition}[section]
\newtheorem{remark}{Remark}[section]
\newtheorem{definition}{Definition}[section]

\newtheorem{thm}[theorem]{\bf Theorem}

\newtheorem{lem}[lemma]{\bf Lemma}
\newtheorem{rmk}[remark]{\bf Remark}
\newtheorem{prop}[proposition]{\bf Proposition}
\newtheorem{defn}[definition]{\bf Definition}

\begin{document}

\allowdisplaybreaks  

\doublespace

\title[wave equation with variable coefficients and supercritical source] {Global solutions and blow-up for the wave equation with variable coefficients: II. boundary supercritical source}

\author[Tae Gab Ha]{Tae Gab Ha}

\address{Department of Mathematics, and Institute of Pure and Applied Mathematics, Jeonbuk National University, Jeonju 54896, Republic of Korea}

\email{tgha@jbnu.ac.kr}

\subjclass[2020]{35L05; 35L20; 35A01; 35B40; 35B44}

\keywords{wave equation with variable coefficients; supercritical source; existence of solutions; energy decay rates; blow-up}


\date{}
\maketitle

\begin{abstract}
In this paper, we consider the wave equation with variable coefficients and boundary damping and supercritical source terms. The goal of this work is devoted to prove the local and global existence, and classify decay rate of energy depending on the growth near zero on the damping term. Moreover, we prove the blow-up of the weak solution with positive initial energy as well as nonpositive initial energy.
\end{abstract}

\section{ Introduction }
\setcounter{equation}{0}

In this paper, we are concerned with the local and global existence, energy decay rates and finite time blow-up of the solution for the following wave equation
\begin{equation}\label{1}
\begin{cases}
\vspace{3mm} u_{tt} - \mu(t) Lu = f(x,t) &\hspace{5mm} \text{in} \hspace{5mm}
\Omega~~\times~~ (0, +\infty),
\\
\vspace{3mm} u = 0 &\hspace{5mm} \text{on} \hspace{5mm} \Gamma_0 ~~\times~~ (0, +\infty),
\\
\vspace{3mm} \mu(t) \frac{\partial u}{\partial \nu_L} + q(u_t) = h(u) &\hspace{5mm} \text{on} \hspace{5mm} \Gamma_1 ~~\times~~ (0, +\infty),
\\
\vspace{3mm} u(x,0) = u_0(x), \hspace{5mm} u_t(x,0) = u_1(x),
\end{cases}
\end{equation}
where $Lu =  div (A(x)\nabla u) =  \sum^n_{i,j=1} \frac{\partial}{\partial x_i} \bigl(a_{ij}(x)\frac{\partial u}{\partial x_j}\bigr)$, where $A(x) = (a_{ij}(x))$ is a symmetric and positive matrix, and $\frac{\partial u}{\partial \nu_L} = \sum^n_{i,j=1} a_{ij}(x)\frac{\partial u}{\partial x_j}\nu_i$, where $\nu = (\nu_1, \cdots, \nu_n)$ is the outward unit normal to $\Gamma$. $\Omega$ is a bounded domain of $\mathbb{R}^n$($n \geq 3$) with smooth boundary $\Gamma =\Gamma_0 \cup \Gamma_1$. Here, $\Gamma_0$ and $\Gamma_1$ are closed and disjoint with $meas(\Gamma_0) \neq 0$.

During the past decades, the problem \eqref{1} has been widely studied. The condition $h(s)s \leq 0$ means that $h$ represents an attractive force. When $h(s)s \geq 0$ as in the present case, $h$ represents a source term. This situation is more delicate than attractive force, since the solution of \eqref{1} can blow up. The damping-source interplay in system \eqref{1} arise naturally in many contexts, for instance, in classical mechanics, fluid dynamics and quantum field theory (cf. \cite{lasi1, sega}). The interaction between two competitive force, that is damping term and source term, make the problem attractive from the mathematical point of view.

For the present case when $h$ is polynomial nonlinear source term such as $h(u) = |u|^\gamma u$, the stability of \eqref{1} has been studied by many authors (see \cite{cava1, cava2, ha1, ha2, ha3, ha4, ha5, ha6, ha7, ha8, ha9, ha10, ha11, ha12, komo, lasi2, lasi3, park1, park2, park3, park4, viti2} and a list on references therein), where $h(u)$ is subcritical or critical source. However, very few results addressed wave equations influenced by supercritical sources (cf. \cite{boci1, boci3, guo1, guo2, viti1}). For example, \cite{viti1} proved the local and global existence, uniqueness and Hadamard well-posedness for the wave equation when source terms can be supercritical or super-supercritical. However, the author do not considered the energy decay and blow-up of the solutions. \cite{guo1} considered a system of nonlinear wave equations with supercritical sources and damping terms. They proved global existence and exponential and algebraic uniform decay rates of energy, moreover, blow-up result for weak solutions with nonnegative initial energy. But as far as I know, the only problem with considering supercritical source is the constant coefficients case, that is, $A = I$ and dimension $n=3$.

In the case of variable coefficients, that is $A \neq I$, boundary stability of the wave equation was considered in \cite{bouk, cava3, ha2, ha5, ha8}. The wave equations with variable coefficients arise in mathematical modeling of inhomogeneous media in solid mechanics, electromagnetic, fluid flows through porous media. For the variable coefficients problem, the main tool is Riemannian geometry method, which was introduced by \cite{yao2} and has been widely used in the literature, see \cite{cao, cava3, lasi4, li, lu, wu1, yao1} and a list of references therein. However, there were very few results considered the source term. For example, \cite{bouk} proved the uniform decay rate of the energy to the viscoelastic wave equation with variable coefficients and acoustic boundary conditions without damping term. Recently, \cite{ha11} studied the general decay rate of the energy for the wave equation with variable coefficients and Balakrishnan-Taylor damping and source term without imposing any restrictive growth near zero on the damping term. However, above mentioned examples were not considered Riemannian geometry, and only treated a subcritical source. More recently, \cite{ha5-1} proved the uniform energy decay rates of the wave equation with variable coefficients applying the Riemannian geometry method and modified multiplier method. But, it was considered a subcritical source. \cite{ha8-1} studied local and global existence, energy decay rate and blow-up of solution for the wave equation with variable coefficients, however it was not considered Reimannian geometry, and was treated the interior supercritical source. There is none, to my knowledge, for the variable coefficients problem having both damping and source terms on Riemannian geometry as well as considering boundary supercritical source.

Our main motivation is constituted by three dimensional case, in which the source term can be supercritical on variable coefficient problem. The differences from previous literatures are as follows:
\begin{enumerate}[(i)]
\item Supercritical source for $n \geq 3$.

\item Variable coefficient problem having source term on Riemannian geometry.

\item Blow-up result with positive initial energy as well as nonpositive initial energy.
\end{enumerate}

In order to overcome difficulties to prove above statements, first, we refine the energy space and a constant used in potential well method, because we do not guarantee $H^1_0(\Omega) \hookrightarrow L^{\gamma+2}(\Gamma_1)$ since the source term is supercritical. Also we have a hypothesis on damping term for proving existence of solutions and energy decay rates (see Remark 2.2). Second, we use the Faedo-Galerkin method because nonlinear semigroup arguments considered in the previous literatures cannot be used since this paper deal with an operator $- \mu(t) L$, which depends on $t$. Third, we refine the key point constants used to prove blow-up result. So, this paper has improved and generalized previous literatures.

The goal of this paper is to prove the existence result using the Faedo-Galerkin method and truncated approximation method, and classify the energy decay rate applying the method developed in \cite{yao2} and \cite{mart}. Moreover, we prove the blow-up of the weak solution with positive initial energy as well as nonpositive initial energy. This paper is organized as follows: In Section 2, we recall the notation, hypotheses and some necessary preliminaries and introduce our main result. In Section 3, we prove the local existence of weak solutions, and show the global existence of weak solution in each two conditions in Section 4. In Section 5, we prove the uniform decay rate under suitable conditions on the initial data and boundary damping by the differential geometric approach. In Section 6, we prove the blow-up of the weak solution with positive initial energy as well as nonpositive initial energy. by using contradiction method.

\section{{\bf  Preliminaries} }
\setcounter{equation}{0}

We begin this section by introducing some notations and our main results. Throughout this
paper, we define the Hilbert space $H^1_0(\Omega) = \{u \in H^1(\Omega) ; u = 0~~ \text{on} ~~ \Gamma_0\}$ with the norm $||u||_{H^1_0(\Omega)} = ||\nabla u||_{L^2(\Omega)}$ and $\mathcal{H} = \{u \in H^1_0(\Omega) ; u \in L^{\gamma+2}(\Gamma_1)\}$ with the norm $||u||_\mathcal{H} = ||u||_{H^1_0(\Omega)} + ||u||_{L^{\gamma+2}(\Gamma_1)}$. $||\cdot||_p$ and $||\cdot||_{p,
\Gamma}$ are denoted by the $L^p(\Omega)$ norm and the $L^p(\Gamma)$ norm, respectively, and $\langle u, v\rangle = \int_\Omega u(x)v(x) dx$ and $\langle u, v\rangle_\Gamma = \int_\Gamma u(x) v(x) d\Gamma$. Moreover, we need some notations on Riemannian geometry, it is mentioned in \cite{yao2} and reference therein. For the reader's comprehension, we will repeat them here.

Let $A(x) = (a_{ij}(x))$ be a symmetric and positive definite matrix for all $x \in \mathbb{R}^n$ $(n \geq 3)$ and $a_{ij}(x)$ be smooth functions on $\mathbb{R}^n$ satisfying
\begin{equation}\label{201}
c_1 \sum^n_{i=1} \xi^2_i \leq \sum^n_{i,j =1} a_{ij}(x) \xi_i \xi_j, \hspace{3mm} \forall x\in \mathbb{R}^n, \hspace{3mm} 0 \neq \xi = (\xi_1, \cdots, \xi_n)^T \in \mathbb{R}^n,
\end{equation}
where $c_1$ is a positive constants. Set
$$
G(x) = \bigl( g_{ij}(x) \bigr) = A^{-1}(x).
$$
For each $x \in \mathbb{R}^n$, we define the inner product $g(\cdot, \cdot) = \langle \cdot, \cdot \rangle_g$ and the norm $|\cdot|_g$ on the tangent space $\mathbb{R}^n_x = \mathbb{R}^n$ by
\begin{equation}\label{202}
g(X, Y) = \langle X, Y \rangle_g = \sum^n_{i,j=1} g_{ij}(x) \alpha_i \beta_j, \hspace{3mm} |X|_g = \langle X, X \rangle^\frac{1}{2}_g, \hspace{3mm} \forall X = \sum^n_{i=1}\alpha_i\frac{\partial}{\partial x_i}, ~~ Y = \sum^n_{i=1}\beta_i\frac{\partial}{\partial x_i} \in \mathbb{R}^n_x.
\end{equation}
Then $(\mathbb{R}^n, g)$ is a Riemannian manifold with Riemann metric $g$. $\nabla_g u$ and $D_g$ are denoted by the gradient of u and Levi-Civita connection in the Riemannian metric $g$, respectively. It follows that
\begin{equation}\label{203}
\nabla_g u = \sum^n_{i=1} \Bigl(\sum^n_{j=1} a_{ij}(x) \frac{\partial u}{\partial x_j} \Bigr) \frac{\partial}{\partial x_i} = A(x) \nabla u, \hspace{3mm} |\nabla_g u|^2_g = \sum^n_{i,j=1} a_{ij}(x) \frac{\partial u}{\partial x_i}\frac{\partial u}{\partial x_j}.
\end{equation}

Let $H$ be a vector field on $(\mathbb{R}^n, g)$. Then the covariant differential $D_gH$ of $H$ determines a bilinear form on $\mathbb{R}^n_x \times \mathbb{R}^n_x$, for each $x \in \mathbb{R}^n$, by
\begin{equation}\label{204}
D_gH(X, Y) = \langle D_{gX} H, Y \rangle_g, \hspace{3mm} \forall X, Y \in \mathbb{R}^n_x,
\end{equation}
where $D_{gX}H$ is the covariant derivative of the vector field $H$ with respect to $X$.

\textbf{$\bf{(H_1)}$ Hypothesis on $\bf\Omega$}.

Let $\Omega\subset\mathbb{R}^n$ be a bounded domain, $n \geq 3$, with smooth boundary $\Gamma = \Gamma_0\cup \Gamma_1$. Here $\Gamma_0$ and $\Gamma_1$ are closed and disjoint with $meas(\Gamma_0) \neq 0$. There exists a vector field $H$ on the Riemannian manifold $(\mathbb{R}^n, g)$ such that
\begin{equation}\label{205}
D_gH(X,X) \geq \sigma |X|^2_g, \hspace{3mm} \forall X \in \mathbb{R}^n_x, \hspace{3mm} x \in \overline{\Omega},
\end{equation}
where $\sigma$ is a positive constant and
\begin{equation}\label{206}
H \cdot \nu \leq 0 \hspace{3mm}\text{on}\hspace{3mm} \Gamma_0 \hspace{3mm}\text{and}\hspace{3mm} H \cdot \nu \geq \delta> 0 \hspace{3mm}\text{on}\hspace{3mm} \Gamma_1,
\end{equation}
where $\nu$ represents the unit outward normal vector to $\Gamma$. Moreover we assume that
\begin{equation}\label{207}
\mu(0)\frac{\partial u_0}{\partial\nu_L} + g(u_1) = h(u_0) \hspace{3mm}
\text{on} \hspace{3mm} \Gamma_1.
\end{equation}

\textbf{$\bf{(H_2)}$ Hypothesis on $\bf\mu$, $\bf f$}.

Let $\mu \in W^{1, \infty}(0,\infty)\cap W^{1,1}(0, \infty)$ satisfying following
conditions:
\begin{equation}\label{208}
\hspace{5mm} \mu(t) \geq \mu_0 > 0 \hspace{5mm} \text{and}
\hspace{5mm} \mu'(t) \leq 0 \hspace{3mm} \text{a.e. in} \hspace{3mm} [0, \infty ),
\end{equation}
where $\mu_0$ is a positive constant. We assume that
\begin{equation}\label{209}
f \in H^1(0,\infty ; L^2(\Omega)).
\end{equation}

\textbf{$\bf(H_3)$ Hypothesis on $\bf q$}.

Let $q : \mathbb{R} \rightarrow \mathbb{R}$ be a nondecreasing $C^1$ function such that $q(0) = 0$ and suppose that there exist positive constants $c_3$, $c_4$, and a strictly increasing and odd function $\beta$ of $C^1$ class on $[-1, 1]$ such that
\begin{equation}\label{210}
|\beta(s)| \leq |q(s)| \leq |\beta^{-1}(s)| \hspace{3mm} \text{if} \hspace{3mm} |s| \leq 1,
\end{equation}
\begin{equation}\label{211}
c_3 |s|^{\rho+1} \leq |q(s)| \leq c_4 |s|^{\rho+1} \hspace{3mm} \text{if} \hspace{3mm} |s| > 1,
\end{equation}
where $\beta^{-1}$ denotes the inverse function of $\beta$, and $\rho \geq \frac{2(n-2)\gamma - 2}{n - (n-2)\gamma}$.

\textbf{$\bf{(H_4)}$ Hypothesis on $\bf\gamma$}.

Let $\gamma$  be a constant satisfying the following condition:
\begin{equation}\label{212}
\frac{1}{n-2} < \gamma \leq \frac{n-1}{n-2}.
\end{equation}

\begin{lem}(\cite{yao2})
Let $u, v \in C^1(\overline{\Omega})$ and $H$, $X$ be vector fields on $(\mathbb{R}^n, g)$. Then
\begin{enumerate}[(i)]
\item
$$
H(u) = \langle \nabla_g u, H\rangle_g, \hspace{3mm} \langle H(x), A(x)X(x)\rangle_g = H(x)\cdot X(x).
$$

\item
$$div(uH) = u ~div(H) + H(u),\hspace{3mm} \int_\Omega div(H) ~dx = \int_\Gamma H\cdot \nu ~d\Gamma.
$$

\item
$$
\int_\Omega u ~div (\nabla_g v) ~dx = \int_\Gamma u \frac{\partial v}{\partial \nu_L} ~d\Gamma - \int_\Omega \langle \nabla_g u, \nabla_g v\rangle_g ~dx.
$$

\item
$$
\langle\nabla_g u, \nabla_g (H(u))\rangle_g = D_g H(\nabla_g u, \nabla_g u) + \frac{1}{2} div (|\nabla_g u|^2_g H) - \frac{1}{2}|\nabla_g u|^2_g div (H).
$$
\end{enumerate}
\end{lem}

\begin{rmk}
Hypothesis \eqref{205} was introduced by Yao \cite{yao2} for the exact controllability of the wave equation with variable coefficients. The existence of such a vector field depends on the sectional curvature of the Riemannian manifold $(\mathbb{R}^n, g)$. There are several methods and examples in \cite{yao2} to find out a vector field $H$ that is satisfied the hypothesis \eqref{205}. Specially, if $A = I$, the constant coefficient case, the condition \eqref{205} is automatically satisfied by choosing $H = x - x_0$ for any $x_0 \in \mathbb{R}^n$.
\end{rmk}

\begin{rmk}
In view of the critical Sobolev imbedding $H^{\frac{1}{2}}(\Gamma_1) \hookrightarrow L^\frac{2(n-1)}{n-2}(\Gamma_1)$, the map $k(u) = |u|^\gamma u$ is not locally Lipschitz from $H^1_0(\Omega)$ into $L^2(\Gamma_1)$ for the supercritical values $\frac{1}{n-2} < \gamma \leq \frac{n-1}{n-2}$. However, by the hypothesis on $\rho$ $\Bigl(\rho \geq \frac{2(n-2)\gamma - 2}{n - (n-2)\gamma}\Bigr)$, $k(u)$ is locally Lipschitz from $H^1_0(\Omega)$ into $L^\frac{\rho+2}{\rho+1}(\Gamma_1)$.
\end{rmk}

\begin{rmk}
For $n > 3$, if $\frac{2}{n-2} \leq \gamma \leq \frac{n-1}{n-2}$, then the inequality $\rho \geq \gamma$ always holds true under the assumption $\rho \geq \frac{2(n-2)\gamma - 2}{n - (n-2)\gamma}$ (see Figure 1).
\end{rmk}

\begin{figure}[ht]
\begin{center}
\resizebox{0.8\textwidth}{!}{%
  \includegraphics{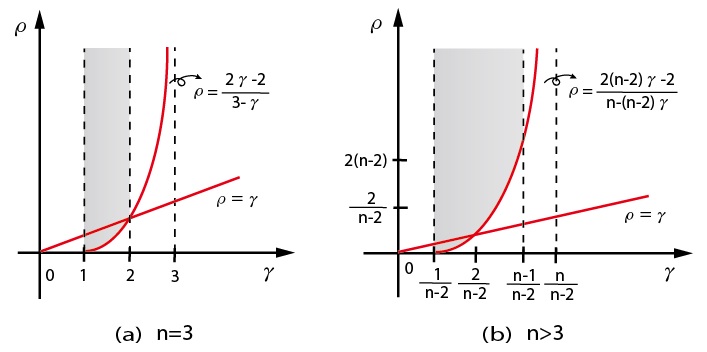}}
\caption{The admissible range of the damping parameter $\rho$ and the exponent of the source $\gamma$.} \label{}
\end{center}
\end{figure}

\begin{defn} (Weak solution).
$u(x,t)$ is called a weak solution of \eqref{1} on $\Omega \times (0,T)$ if $u \in C(0,T ; \mathcal{H}) \cap C^1 (0,T ; L^2(\Omega))$, $u_t |_{\Gamma_1} \in L^{\rho+2}(0, T ; \Gamma_1)$ and satisfies \eqref{1} in the distribution sense, i.e.,
\begin{multline*}
\int^T_0\int_\Omega -u_t \phi_t dxdt + \int_\Omega u_t \phi dx \Bigr|^T_0 + \int^T_0 \mu(t) \int_\Omega \langle \nabla_g u, \nabla_g \phi \rangle_g dx dt
\\
+ \int^T_0\int_{\Gamma_1} q(u_t) \phi d\Gamma dt -  \int^T_0\int_{\Gamma_1} h(u) \phi d\Gamma dt = \int^T_0\int_\Omega f(x,t) \phi dx dt,
\end{multline*}
for any $\phi \in C(0,T ; \mathcal{H}) \cap C^1 (0,T ; L^2(\Omega))$, $\phi |_{\Gamma_1} \in L^{\rho+2}(0, T ; \Gamma_1)$ and $u(x,0) = u_0(x) \in \mathcal{H}$, $u_t(x,0) = u_1(x) \in L^2(\Omega)$.

\end{defn}

\begin{rmk}

One easily check that $\mathcal{H} = H^1_0(\Omega)$ when $\frac{1}{n-2} < \gamma \leq \frac{2}{n-2}$. Moreover, if $n = 3$, then we can replace $\mathcal{H}$ by $H^1_0(\Omega)$, since $H^1_0(\Omega) \hookrightarrow L^{\gamma+2}(\Gamma_1)$ (see Figure 2).

\end{rmk}

\begin{figure}[ht]
\begin{center}
\resizebox{0.5\textwidth}{!}{%
  \includegraphics{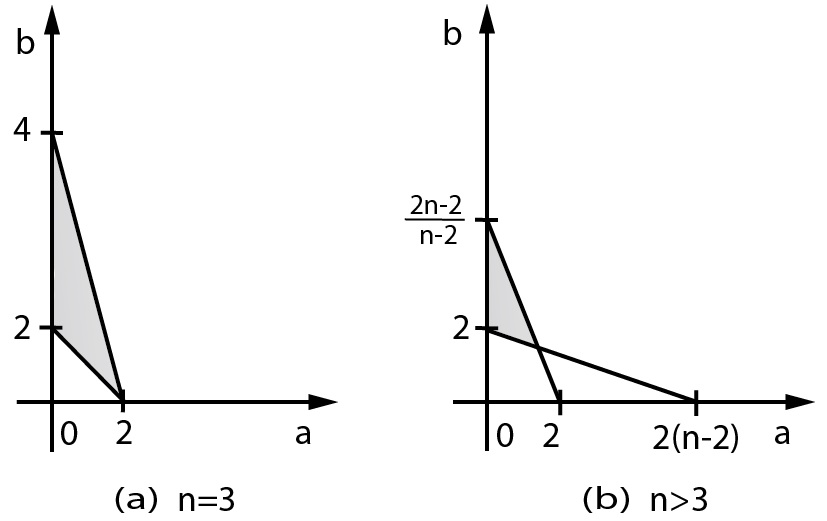}}
\caption{The admissible range of parameters $a$ and $b$ with respect to the trace imbedding $H^1_0(\Omega) \hookrightarrow L^{a\gamma + b}(\Gamma_1)$.}\label{}
\end{center}
\end{figure}

The energy associated to the problem \eqref{1} when $h(u) = |u|^\gamma u$ is given by
$$
E(t) = \frac{1}{2}||u_t(t)||^2_2 + \frac{1}{2}\mu(t)||~|\nabla_g u(t)|_g ||^2_2 - \frac{1}{\gamma+2}||u(t)||^{\gamma+2}_{\gamma+2, \Gamma_1}.
$$

We now state our main results.

\begin{thm}
Suppose that $(H_1)-(H_4)$ hold and let $h(u) = |u|^\gamma u$. Then given the initial data $(u_0,u_1) \in \mathcal{H} \times L^2(\Omega)$, there exist $T>0$ and a weak solution of problem \eqref{1}. Moreover, the following energy identity holds for all $0 \leq t \leq T$:
\begin{equation}\label{213}
E(t) + \int^t_0\int_{\Gamma_1} q(u_s)u_s d\Gamma ds - \frac{1}{2}\int^t_0 \mu'(s) ||~|\nabla_g u(s)|_g ||^2_2 ds - \int^t_0 \int_\Omega f(x,s) u_s dx ds = E(0).
\end{equation}

Furthermore, if one of the assumptions hold: $\rho \geq \gamma$ or
\begin{equation}\label{214}
E(0) < d_0 \hspace{3mm}\text{and}\hspace{3mm} ||~|\nabla_g u_0|_g ||_2 < \lambda_0
\end{equation}
with $f=0$, where
$$
\lambda_0 = \Bigl(\frac{\mu_0}{K^{\gamma+2}_0}\Bigr)^{1/\gamma} \hspace{3mm}\text{and}\hspace{3mm} d_0 = \frac{\gamma\mu_0}{2(\gamma+2)} \lambda^2_0, \hspace{3mm} K_0 = \sup_{u \in \mathcal{H}, u \neq 0} \Bigl(\frac{||u||_{\gamma+2, \Gamma_1}}{||~|\nabla_g u|_g ||_2} \Bigr).
$$
Then the solution $u(x,t)$ of \eqref{1} is global.

\end{thm}



\begin{thm}

Suppose that the hypotheses in Theorem 2.1 and \eqref{214} with $\rho \leq \gamma$ and $f =0$ hold. Then we have following energy decay rates:
\begin{enumerate}[(i)]

\item \emph{\textbf{Case 1}} : $\beta$ is linear. Then we have

$$
E(t) \leq C_1 e^{-\omega t},
$$
where $\omega$ is a positive constant.

\item \emph{\textbf{Case 2}} : $\beta$ has polynomial growth near zero, that is, $\beta(s) = s^{\rho+1}$. Then we have

$$
E(t) \leq \frac{C_2}{(1+t)^\frac{2}{\rho}}.
$$

\item \emph{\textbf{Case 3}} : $\beta$ does not necessarily have polynomial growth near zero. Then we have

$$
E(t) \leq C_3 \Bigl(F^{-1}\Bigl(\frac{1}{t}\Bigr) \Bigr)^2,
$$
where $F(s) = s \beta(s)$ and $C_i$ ($i = 1, 2, 3$) are positive constants that depends only on $E(0)$.

\end{enumerate}

\end{thm}

\begin{thm}
Suppose that hypotheses $(H_1)-(H_4)$ with $\rho < \gamma$ and $f=0$ hold. Moreover, assume that
$$
(u_0, u_1) \in \{(u_0,u_1) \in \mathcal{H}\times L^2(\Omega); ||~|\nabla_g u_0|_g ||_2 > \lambda_0,~ -1 < E(0) < d_0 \}
$$
and
\begin{equation}\label{215}
\beta^{-1}(1) \leq \Bigl( \frac{(\gamma+2) (\mu_0 \gamma \lambda^2_0 - 2 (\gamma+2)E_1)^2}{8(\gamma+1)meas(\Gamma_1) (\mu_0\lambda^2_0 - 2E_1)} \Bigr)^\frac{\gamma+1}{\gamma+2},
\end{equation}
where
$$
E_1 =
\begin{cases}
0 \hspace{3mm}&\text{if}\hspace{3mm} E(0) < 0,
\\
\text{positive constant satisfying}\hspace{3mm} E(0) < E_1 < d_0 \hspace{3mm}\text{and}\hspace{3mm} E_1 < E(0) + 1 \hspace{3mm}&\text{if}\hspace{3mm} E(0) \geq 0.
\end{cases}
$$
Then the weak solution of the problem \eqref{1} blows up in finite time.
\end{thm}

\section{ {\bf Proof of Theorem 2.1 : local existence} }
\setcounter{equation}{0}

This section is devoted to prove the local existence in Theorem 2.1. The proof is based on three steps according to the following condition of the source term $h$:
\begin{enumerate}
\item Existence of the global solution when $h$ is globally Lipschitz from $H^1_0(\Omega)$ to $L^2(\Gamma_1)$.
\item Existence of the local solution when $h$ is locally Lipschitz from $H^1_0(\Omega)$ to $L^2(\Gamma_1)$.
\item Existence of the local solution when $h$ is locally Lipschitz from $H^1_0(\Omega)$ to $L^\frac{\rho+2}{\rho+1}(\Gamma_1)$.
\end{enumerate}

Then since the mapping $|u|^\gamma u$ is locally Lipschitz from $H^1_0(\Omega)$ to $L^\frac{\rho+2}{\rho+1}(\Gamma_1)$ (see Remark 2.2), the existence of the local solution can be guaranteed even if $h(u) = |u|^\gamma u$.

\subsection{Globally Lipschitz source}

We first deal with the case where the source $h$ is globally Lipschitz from $H^1_0(\Omega)$ to $L^2(\Gamma_1)$. In this case, we have the following result.

\begin{prop}
Assume that $(H_1)-(H_3)$ hold. In addition, assume that  $(u_0,u_1) \in \mathcal{H} \times L^2(\Omega)$ and $h : H^1_0(\Omega) \rightarrow L^2(\Gamma_1)$ is globally Lipschitz continuous. Then problem \eqref{1} has a unique global solutions $u \in C(0,T ; \mathcal{H}) \cap C^1 (0,T ; L^2(\Omega))$ for arbitrary $T > 0$.
\end{prop}

Our goal in this subsection is to show the local existence result for problem \eqref{1}. We construct an approximate solution by using the Faedo-Galerkin method. Let $\{w_j\}_{j \in \mathbb{N}}$ be a basis in $H^1_0(\Omega)$ and define $V_m = span\{w_1, w_2, \cdots, w_m \}$. Let $u^m_0$ and $u^m_1$ be sequences of $V_m$ such that $u^m_0 \rightarrow u_0$ strongly in $H^1_0(\Omega)$ and $u^m_1 \rightarrow u_1$ strongly in $L^2(\Omega)$. We search for a function, for each $\eta \in (0, 1)$ and $m \in \mathbb{N}$,
\begin{equation*}
u^{\eta m}(t) = \sum^m_{j=1} \delta^{jm}(t)w_j
\end{equation*}
satisfying the approximate perturbed equation
\begin{equation}\label{301}
\begin{cases}
\vspace{3mm} \int_\Omega u^{\eta m}_{tt} w dx + \mu(t) \int_\Omega \langle \nabla_g u^{\eta m}, \nabla_g w \rangle_g dx + \eta \int_{\Gamma_1} u^{\eta m}_t w d\Gamma
\\
\vspace{3mm} \hspace{10mm} + \int_{\Gamma_1} q(u^{\eta m}_t) w d\Gamma - \int_{\Gamma_1} h(u^{\eta m}) w d\Gamma = \int_\Omega f(t) w dx,\hspace{3mm} w \in V_m,
\\
\vspace{3mm} u^{\eta m}_0 = \sum^m_{j=1} \langle u_0, w_j \rangle w_j, \hspace{3mm} u^{\eta m}_1 = \sum^m_{j=1} \langle u_1, w_j \rangle w_j.
\end{cases}
\end{equation}

Since \eqref{301} is a normal system of ordinary differential equations, there exist $u^{\eta m}$, solutions to problem \eqref{301}. A solution $u$ to problem \eqref{1} on some internal $[0, t_m)$, $t_m \in (0, T]$ will be obtain as the limit of $u^{\eta m}$ as $m \rightarrow \infty$ and $\eta \rightarrow 0$. Next, we show that $t_m = T$ and the local solution is uniformly bounded independent of $m$, $\eta$ and $t$. For this purpose, let us replace $w$ by $u^{\eta m}_t$ in \eqref{301} we obtain
\begin{equation}\label{302}
\begin{aligned}
&\frac{d}{dt} \Biggl[ \frac{1}{2}||u^{\eta m}_t||^2_2 + \frac{1}{2}\mu(t)|| ~|\nabla_g u^{\eta m}|_g||^2_2 + \frac{1}{\gamma+2}||u^{\eta m}||^{\gamma+2}_{\gamma+2, \Gamma_1} \Biggr] + \eta ||u^{\eta m}_t||^2_{2,\Gamma_1} + \int_{\Gamma_1} q(u^{\eta m}_t)~u^{\eta m}_t d\Gamma
\\
& = \frac{1}{2} \mu'(t) || ~|\nabla_g u^{\eta m}|_g||^2_2 + \int_{\Gamma_1} h(u^{\eta m}) u^{\eta m}_t d\Gamma +  \int_{\Gamma_1} |u^{\eta m}|^\gamma u^{\eta m} u^{\eta m}_t d\Gamma + \int_\Omega f(t) u^{\eta m}_t dx.
\end{aligned}
\end{equation}

We will now estimate $\int_{\Gamma_1} q(u^{\eta m}_t)~u^{\eta m}_t d\Gamma$, $\int_{\Gamma_1} |u^{\eta m}|^\gamma u^{\eta m} u^{\eta m}_t d\Gamma$ and $\int_{\Gamma_1} h(u^{\eta m}) u^{\eta m}_t d\Gamma$. From the hypotheses on $q$, we have
\begin{equation}\label{303}
\begin{aligned}
\int_{\Gamma_1} q(u^{\eta m}_t)~u^{\eta m}_t d\Gamma & = \int_{|u^{\eta m}_t| \leq 1} q(u^{\eta m}_t)~u^{\eta m}_t d\Gamma + \int_{|u^{\eta m}_t| > 1} q(u^{\eta m}_t)~u^{\eta m}_t d\Gamma
\\
& \geq \int_{|u^{\eta m}_t(t)| > 1} q(u^{\eta m}_t)~u^{\eta m}_t d\Gamma
\\
& \geq c_3 \int_{\Gamma_1} |u^{\eta m}_t|^{\rho+2} d\Gamma - c_3 \int_{|u^{\eta m}_t| \leq 1} |u^{\eta m}_t|^{\rho+2} d\Gamma
\\
& \geq c_3 ||u^{\eta m}_t||^{\rho+2}_{\rho+2, \Gamma_1} - c_3 meas(\Gamma_1).
\end{aligned}
\end{equation}

From the assumption $\rho \geq \frac{2(n-2)\gamma - 2}{n - (n-2)\gamma}$, we have the imbedding $H^1_0(\Omega)\hookrightarrow L^\frac{2(n-1)}{n-2}(\Gamma_1) \hookrightarrow L^{(\gamma+1)\frac{\rho+2}{\rho+1}}(\Gamma_1)$, so that by Young's inequality with $\frac{\rho+1}{\rho+2} + \frac{1}{\rho+2} = 1$ we deduce that
\begin{equation}\label{304*}
\begin{aligned}
\int_{\Gamma_1} |u^{\eta m}|^\gamma u^{\eta m} u^{\eta m}_t d\Gamma &\leq C(\epsilon_0) ||u^{\eta m}||^{\gamma+1}_{(\gamma+1)\frac{\rho+2}{\rho+1}, \Gamma_1} + \epsilon_0 ||u^{\eta m}_t||^{\rho+2}_{\rho+2, \Gamma_1}
\\
& \leq C(\epsilon_0) || ~|\nabla_g u^{\eta m}|_g||^{\gamma+1}_2 + \epsilon_0 ||u^{\eta m}_t||^{\rho+2}_{\rho+2, \Gamma_1}
\\
&\leq C(\epsilon) \bigl(1 + || ~|\nabla_g u^{\eta m}|_g||_2\bigr)^2 + \epsilon_0 ||u^{\eta m}_t||^{\rho+2}_{\rho+2, \Gamma_1}
\\
&\leq C(\epsilon) \bigl(1 + || ~|\nabla_g u^{\eta m}|_g||^2_2\bigr) + \epsilon_0 ||u^{\eta m}_t||^{\rho+2}_{\rho+2, \Gamma_1}.
\end{aligned}
\end{equation}

Under the assumption that $h$ is globally Lipschitz from $H^1_0(\Omega)$ into $L^2(\Gamma_1)$ we have
$$
||h(u^{\eta m})||_{2,\Gamma_1} \leq ||h(u^{\eta m}) - h(0)||_{2,\Gamma_1} + ||h(0)||_{2, \Gamma_1} \leq L_h||\nabla u^{\eta m}||_2 + ||h(0)||_{2,\Gamma_1} \leq C_4 (||\nabla u^{\eta m}||_2 + 1),
$$
where $L_h$ is the Lipschitz constant and $C_4$ is for some positive constant, so that by H\"{o}lder's and Young's inequalities and from the fact \eqref{201} and the imbedding $L^{\rho+2}(\Gamma_1) \hookrightarrow L^2(\Gamma_1)$, we deduce that
\begin{equation}\label{304}
\begin{aligned}
\Bigl|\int_{\Gamma_1} h(u^{\eta m}) u^{\eta m}_t d\Gamma \Bigr|&\leq C(\epsilon)||h(u^{\eta m})||^2_{2,\Gamma_1} + \epsilon ||u^{\eta m}_t||^2_{2,\Gamma_1}
\\
&\leq C(\epsilon) \Bigl(\frac{1}{c_1}  || ~|\nabla_g u^{\eta m}|_g||^2_2 + 1\Bigr) + \epsilon ||u^{\eta m}_t||^2_{2,\Gamma_1},
\\
&\leq  C(\epsilon) \Bigl(\frac{1}{c_1}  || ~|\nabla_g u^{\eta m}|_g||^2_2 + 1\Bigr) + \epsilon C^2_{\rho+2,2}||u^{\eta m}_t||^2_{\rho+2,\Gamma_1}
\\
&\leq C(\epsilon) \Bigl(\frac{1}{c_1}  || ~|\nabla_g u^{\eta m}|_g||^2_2 + 1\Bigr) + \epsilon 2^{\rho+1} C^2_{\rho+2,2}  \bigl(1 + ||u^{\eta m}_t||^{\rho+2}_{\rho+2,\Gamma_1}\bigr),
\end{aligned}
\end{equation}
where $C_{\rho+2,2}$ is an imbedding constant.

Replacing \eqref{303}, \eqref{304*} and \eqref{304} in \eqref{302} we get
\begin{equation}\label{305}
\begin{aligned}
&\frac{d}{dt} \Biggl[ \frac{1}{2}||u^{\eta m}_t||^2_2 + \frac{1}{2}\mu(t)|| ~|\nabla_g u^{\eta m}|_g||^2_2 + \frac{1}{\gamma+2}||u^{\eta m}||^{\gamma+2}_{\gamma+2, \Gamma_1} \Biggr]
\\
& \hspace{5mm} + \eta ||u^{\eta m}_t||^2_{2,\Gamma_1} + \Bigl( c_3 - \epsilon_0 - \epsilon 2^{\rho+1} C^2_{\rho+2,2} \Bigr)
||u^{\eta m}_t||^{\rho+2}_{\rho+2, \Gamma_1}
\\
&\leq c_3 meas(\Gamma_1) + C(\epsilon_0) + C(\epsilon) +  \epsilon 2^{\rho+1} C^2_{\rho+2,2}
\\
&\hspace{5mm} + \Bigl(\frac{1}{2}\mu'(t) + C(\epsilon_0) + \frac{C(\epsilon)}{c_1}  \Bigr) || ~|\nabla_g u^{\eta m}|_g||^2_2  + \frac{1}{2}||f(t)||^2_2 + \frac{1}{2}||u^{\eta m}_t||^2_2.
\end{aligned}
\end{equation}
By integrating \eqref{305} over $(0,t)$ with $t \in (0, t_m)$ we have
\begin{equation}\label{306}
\begin{aligned}
&\frac{1}{2}||u^{\eta m}_t||^2_2 + \frac{1}{2}\mu(t)|| ~|\nabla_g u^{\eta m}|_g||^2_2 + \frac{1}{\gamma+2}||u^{\eta m}||^{\gamma+2}_{\gamma+2, \Gamma_1} + \eta \int^t_0 ||u^{\eta m}_s||^2_{2,\Gamma_1} ds
\\
& \hspace{5mm}+ \Bigl( c_3 -- \epsilon_0 - \epsilon 2^{\rho+1} C^2_{\rho+2,2} \Bigr) \int^t_0 ||u^{\eta m}_s||^{\rho+2}_{\rho+2, \Gamma_1} ds
\\
& \leq \Bigl(c_3 meas(\Gamma_1)+ C(\epsilon_0) + C(\epsilon) +  \epsilon 2^{\rho+1} C^2_{\rho+2,2} \Bigr)T
\\
&\hspace{5mm} + \frac{1}{2} ||u_1||^2_2 + \frac{1}{2} \mu(0)|| ~|\nabla_g u_0|_g||^2_2 + \frac{1}{\gamma+2}||u_0||^{\gamma+2}_{\gamma+2, \Gamma_1}
\\
& \hspace{5mm} + \Bigl(\frac{1}{2} ||\mu'||_{L^\infty(0,T)} + C(\epsilon_0) + \frac{C(\epsilon)}{c_1} \Bigr) \int^t_0 || ~|\nabla_g u^{\eta m}(s)|_g||^2_2 ds
\\
&\hspace{5mm} + \frac{1}{2} \int^t_0 ||f(s)||^2_2 ds + \frac{1}{2}\int^t_0 ||u^{\eta m}_s(s)||^2_2 ds.
\end{aligned}
\end{equation}
Therefore, choosing $\epsilon_0 = \frac{1}{4}c_3$ and $\epsilon = \frac{c_3}{2^{\rho+3}C^2_{\rho+2,2}}$ and then by Gronwall's lemma we obtain
\begin{equation}\label{307}
||u^{\eta m}_t||^2_2 +|| ~|\nabla_g u^{\eta m}|_g||^2_2 + ||u^{\eta m}||^{\gamma+2}_{\gamma+2, \Gamma_1} + \int^t_0 ||u^{\eta m}_s(s)||^2_{2,\Gamma_1} ds  + \int^t_0 ||u^{\eta m}_s(s)||^{\rho+2}_{\rho+2, \Gamma_1} ds \leq C_5,
\end{equation}
where $C_5$ is a positive constant which is independent of $m$, $\eta$ and $t$. The estimate \eqref{307} implies that
\begin{equation}\label{308}
u^{\eta m} \hspace{3mm}\text{is uniformly bounded in} \hspace{3mm} L^\infty (0, T ; \mathcal{H})
\end{equation}
and
\begin{equation}\label{309}
u^{\eta m}_t \hspace{3mm}\text{is uniformly bounded in} \hspace{3mm} L^\infty (0, T ; L^2(\Omega)).
\end{equation}
We note that from \eqref{307}, taking the hypotheses on $q$ into account we also obtain
\begin{equation}\label{310}
\int^t_0\int_{\Gamma_1} |q(u^{\eta m}_s(s))|^2 d\Gamma ds \leq C_6,
\end{equation}
where $C_6$ is a positive constant independent of $m$, $\eta$ and $t$.

From \eqref{307}-\eqref{310}, there exists a subsequence of $\{u^{\eta m}\}$, which we still denote by $\{u^{\eta m}\}$, such that
$$
u^{\eta m} \rightarrow u^\eta \hspace{3mm}\text{weakly star in}\hspace{3mm} L^\infty (0, T ; \mathcal{H}),
$$
$$
u^{\eta m}_t \rightarrow u^\eta_t \hspace{3mm}\text{weakly star in}\hspace{3mm} L^\infty (0, T ; L^2(\Omega)),
$$
$$
u^{\eta m} \rightarrow u^\eta \hspace{3mm}\text{weakly in}\hspace{3mm} L^2 (0, T ; \mathcal{H}),
$$
$$
u^{\eta m}_t \rightarrow u^\eta_t \hspace{3mm}\text{weakly in}\hspace{3mm} L^2 (0, T ; L^2(\Omega)),
$$
$$
u^{\eta m}_t \rightarrow u^\eta_t \hspace{3mm}\text{weakly in}\hspace{3mm} L^2 (0, T ; L^2(\Gamma_1)),
$$
$$
u^{\eta m}_{tt} \rightarrow u^\eta_{tt} \hspace{3mm}\text{weakly in}\hspace{3mm} L^2 (0, T ; H^{-1}(\Omega)),
$$
$$
q(u^{\eta m}_t) \rightarrow \psi \hspace{3mm}\text{weakly in}\hspace{3mm} L^2 (0, T ; L^2(\Gamma_1)).
$$
Since $H^{1/2}(\Gamma) \hookrightarrow L^2(\Gamma)$ is compact, we have, thanks to Aubin-Lions Theorem that
$$
u^{\eta m} \rightarrow u^\eta \hspace{3mm}\text{strongly in}\hspace{3mm} L^2 (0, T ; L^2(\Gamma_1)),
$$
and consequently, by making use of Lions lemma, we deduce
$$
h(u^{\eta m}) \rightarrow h(u^\eta) \hspace{3mm}\text{weakly in}\hspace{3mm} L^2 (0, T ; L^2(\Gamma_1)).
$$

The above convergences permit us to pass to the limit in the \eqref{301}. Since $\{w_j\}$ is a basis of $H^1_0(\Omega)$ and $V_m$ is dense in $H^1_0(\Omega)$, after passing to the limit we obtain
\begin{equation}\label{311}
\begin{aligned}
&\int^T_0 \int_\Omega u^\eta_{tt}, v dx \theta(t) dt + \int^T_0 \mu(t) \int_\Omega \langle \nabla_g u^\eta, \nabla_g v \rangle_g dx \theta(t) dt + \eta \int^T_0 \int_{\Gamma_1} u^\eta_t v d\Gamma \theta(t) dt
\\
&+ \int^T_0\int_{\Gamma_1} \psi v d\Gamma \theta(t) dt - \int^T_0\int_{\Gamma_1} h(u^\eta) v dx \theta(t) dt = \int^T_0\int_\Omega f(x,t) v dx \theta(t) dt,
\end{aligned}
\end{equation}
for all $\theta \in D(0,T)$ and $v \in H^1_0(\Omega)$.

Since estimates \eqref{307} and \eqref{310} are also independent of $\eta$, we can pass to the limit when $ \eta \rightarrow 0$ in $u^\eta$ obtaining a function $u$ by the same argument used to obtain
$u^\eta$ from $u^{\eta m}$, such that
\begin{equation}\label{312}
u^\eta \rightarrow u \hspace{3mm}\text{weakly in}\hspace{3mm} L^2 (0, T ; \mathcal{H}),
\end{equation}
\begin{equation}\label{313}
u^\eta_t \rightarrow u_t \hspace{3mm}\text{weakly in}\hspace{3mm} L^2 (0, T ; L^2(\Omega)),
\end{equation}
\begin{equation}\label{314}
u^\eta_t \rightarrow u_t \hspace{3mm}\text{weakly in}\hspace{3mm} L^2 (0, T ; L^2(\Gamma_1)),
\end{equation}
\begin{equation}\label{315}
u^\eta_{tt} \rightarrow u_{tt} \hspace{3mm}\text{weakly in}\hspace{3mm} L^2 (0, T ; H^{-1}(\Omega)),
\end{equation}
\begin{equation}\label{316}
q(u^\eta _t) \rightarrow \psi \hspace{3mm}\text{weakly in}\hspace{3mm} L^2 (0, T ; L^2(\Gamma_1)),
\end{equation}
\begin{equation}\label{317}
h(u_\eta) \rightarrow h(u)  \hspace{5mm}\text{weakly in} \hspace{5mm} L^2(0,T ; L^2(\Gamma_1)).
\end{equation}

By above convergences in \eqref{311}, we have
\begin{equation}\label{318}
\begin{aligned}
&\int^T_0 \int_\Omega u_{tt}, v dx \theta(t) dt + \int^T_0 \mu(t) \int_\Omega \langle\nabla_g u, \nabla_g v \rangle_g dx \theta(t) dt
\\
&+ \int^T_0\int_{\Gamma_1} \psi v d\Gamma \theta(t) dt - \int^T_0\int_{\Gamma_1} h(u) v dx \theta(t) dt = \int^T_0\int_\Omega f(x,t) v dx \theta(t) dt.
\end{aligned}
\end{equation}

From the \eqref{318} and taking $v \in D(\Omega)$, we conclude that
\begin{equation}\label{319}
u_{tt} - \mu(t) Lu  = f  \hspace{3mm} \text{in} \hspace{3mm} D'(\Omega \times (0,T))
\end{equation}
and since \eqref{316} and \eqref{317}, it holds that
$$
\mu(t) \frac{\partial u}{\partial \nu} + \psi = h(u) \hspace{3mm} \text{in} \hspace{3mm}  L^2(0, T; L^2(\Gamma_1)).
$$

Our goal is to show that $\psi = q(u_t)$. Indeed, considering $w= u^{\eta m}$ in \eqref{301} and then integrating over $(0,T)$, we have
\begin{multline*}
\int^T_0 \langle u^{\eta m}_{tt}, u^{\eta m} \rangle dt + \int^T_0 \mu(t) ||~|\nabla_g u^{\eta m}|_g||^2_2 dt + \eta \int^T_0 \langle u^{\eta m}_t, u^{\eta m}\rangle_{\Gamma_1} dt
\\
+ \int^T_0 \langle q(u^{\eta m}_t), u^{\eta m}\rangle_{\Gamma_1} dt - \int^T_0 \langle h(u^{\eta m}), u^{\eta m}\rangle_{\Gamma_1} dt = \int^T_0 \langle f, u^{\eta m} \rangle dt.
\end{multline*}
Then from convergences \eqref{312}-\eqref{317} we obtain
\begin{equation}\label{320}
\lim_{m\rightarrow \infty, \eta \rightarrow 0} \int^T_0 \mu(t) ||~|\nabla_g u^{\eta m}|_g||^2_2 dt = - \int^T_0 \langle u_{tt}, u \rangle dt - \int^T_0 \langle \psi, u \rangle_{\Gamma_1} dt + \int^T_0 \langle h(u), u \rangle_{\Gamma_1} dt + \int^T_0 \langle f, u \rangle dt.
\end{equation}
By combining \eqref{319} and \eqref{320}, we have
$$
\lim_{m\rightarrow \infty, \eta \rightarrow 0} \int^T_0 \mu(t) ||~|\nabla_g u^{\eta m}|_g||^2_2 dt = \int^T_0 \mu(t) ||~|\nabla_g u|_g||^2_2 dt,
$$
which implies that
\begin{equation}\label{321}
|\nabla_g u^{\eta m}|_g \rightarrow |\nabla_g u|_g \hspace{3mm}\text{strongly in}\hspace{3mm} L^2(0, T; L^2(\Omega)).
\end{equation}

Next, considering $w = u^{\eta m}_t$ in \eqref{301} and then integrating over $(0,T)$, we have
\begin{multline*}
\int^T_0 \langle u^{\eta m}_{tt}, u^{\eta m}_t \rangle dt + \int^T_0 \mu(t) \int_\Omega \langle \nabla_g u^{\eta m}, \nabla_g u^{\eta m}_t \rangle_g dx dt + \eta \int^T_0 ||u^{\eta m}_t||^2_{2,\Gamma_1} dt
\\
+ \int^T_0 \langle q(u^{\eta m}_t), u^{\eta m}_t\rangle_{\Gamma_1} dt - \int^T_0 \langle h(u^{\eta m}), u^{\eta m}_t \rangle_{\Gamma_1}  dt = \int^T_0 \langle f, u^{\eta m}_t \rangle dt.
\end{multline*}
From \eqref{313}-\eqref{317} and \eqref{321}, we arrive at
\begin{equation}\label{322}
\lim_{m\rightarrow \infty, \eta \rightarrow 0} \int^T_0 \langle q(u^{\eta m}_t), u^{\eta m}_t\rangle_{\Gamma_1} dt = \int^T_0 \langle \psi, u_t \rangle_{\Gamma_1} dt.
\end{equation}

On the other hand, since $q$ is a nondecreasing monotone function, we get
$$
\int^T_0 \langle q(u^{\eta m}_t) - q(\varphi), u^{\eta m}_t - \varphi \rangle_{\Gamma_1} dt \geq 0
$$
for all $\varphi \in L^2(\Gamma_1)$. Thus, it implies that
$$
\int^T_0 \langle q(u^{\eta m}_t), \varphi \rangle_{\Gamma_1} dt + \int^T_0 \langle q(\varphi), u^{\eta m}_t - \varphi \rangle_{\Gamma_1} dt \leq \int^T_0 \langle q(u^{\eta m}_t), u^{\eta m}_t\rangle_{\Gamma_1} dt.
$$
By considering \eqref{314}, \eqref{316} and \eqref{322}, we obtain
\begin{equation}\label{320*}
\int^T_0 \langle \psi - q(\varphi), u_t - \varphi \rangle_{\Gamma_1} dt \geq 0,
\end{equation}
which implies that $\psi = q(u_t)$.

We now show the uniqueness of the solution. Let $u^1$ and $u^2$ be two solutions of problem \eqref{1}. Then $z = u^1 - u^2$ verifies
$$
\int_\Omega z_{tt} w dx +  \mu(t) \int_\Omega \langle \nabla_g z, \nabla_g w \rangle_g dx + \int_{\Gamma_1} (q(u^1_t) - q(u^2_t)) w d\Gamma =  \int_{\Gamma_1} (h(u^1) - h(u^2) w d\Gamma dt,
$$
for all $w \in\mathcal{H}$. By replacing $w = z_t$ in above identity and observing that $q$ is monotonously nondecreasing and $h : H^1_0(\Omega) \rightarrow L^2(\Gamma_1)$ is globally Lipschitz, it holds that
$$
\frac{d}{dt} \Biggl[ \frac{1}{2}||z_t||^2_2 + \frac{1}{2}\mu(t)||~|\nabla_g z|_g||^2_2 \Biggr]   \leq C_7 ||~|\nabla_g z|_g||^2_2,
$$
Where $C_7$ is for some positive constant. By integrating from $0$ to $t$ and using Gronwall's Lemma, we conclude that $||z_t||_2 = ||~|\nabla_g z|_g||_2 = 0$.

\subsection{Locally Lipschitz source}

In this subsection, we loosen the globally Lipschitz condition on the source by allowing $h$ to be locally Lipschitz continuous. More precisely, we have the following result.

\begin{prop}
Assume that $(H_1)-(H_4)$ hold. In addition, assume that  $(u_0,u_1) \in \mathcal{H} \times L^2(\Omega)$ and $h : H^1_0(\Omega) \rightarrow L^2(\Gamma_1)$ is locally Lipschitz continuous satisfying $c_5|s|^{\gamma+1} \leq |h(s)| \leq c_6|s|^{\gamma+1}$, where $c_5$, $c_6$ are for some positive constants. Then problem \eqref{1} has a unique local solution $u \in C(0,T ; \mathcal{H}) \cap C^1 (0,T ; L^2(\Omega))$ for some $T > 0$.
\end{prop}

\begin{proof}

Define
$$
h_K(u) =
\begin{cases}
h(u) \hspace{3mm}&\text{if}\hspace{3mm} ||~|\nabla_g u|_g||_2 \leq K,
\\
h\Bigl(\frac{Ku}{||~|\nabla_g u|_g||_2}\Bigr) \hspace{3mm}&\text{if}\hspace{3mm} ||~|\nabla_g u|_g||_2 > K,
\end{cases}
$$
where $K$ is a positive constant. With this truncated $h_K$, we consider the following problem:
\begin{equation}\label{321*}
\begin{cases}
\vspace{3mm} u_{tt} - \mu(t) Lu = f(x,t) &\hspace{5mm} \text{in} \hspace{5mm}
\Omega~~\times~~ (0, +\infty),
\\
\vspace{3mm} u = 0 &\hspace{5mm} \text{on} \hspace{5mm} \Gamma_0 ~~\times~~ (0, +\infty),
\\
\vspace{3mm} \mu(t) \frac{\partial u}{\partial \nu_L} + q(u_t) = h_K(u) &\hspace{5mm} \text{on} \hspace{5mm} \Gamma_1 ~~\times~~ (0, +\infty),
\\
\vspace{3mm} u(x,0) = u_0(x), \hspace{5mm} u_t(x,0) = u_1(x).
\end{cases}
\end{equation}

Since $h_K : H^1_0(\Omega) \rightarrow L^2(\Gamma_1)$ is globally Lipschitz continuous for each $K$ (see \cite{chue}), then by Proposition 3.1, the truncated problem \eqref{321*} has a unique global solution $u_K \in C(0,T ;\mathcal{H}) \cap C^1 (0,T ; L^2(\Omega))$ for any $T > 0$. Moreover by \cite{lasi2} there exists a sequence of functions $u^l_K$, which converges to $u_K$ in the class $C(0,T ; H^2(\Omega)) \cap C^1 (0,T ; H^1_0(\Omega))$. For simplifying the notation in the rest of the proof, we shall express $u^l_K$ as $u$.

By the regularity of $u$, we can multiply $\eqref{321*}$ by $u_t$ and integrate on $\Omega \times (0,t)$, where $0 < t < T$. Then we obtain by using the fact $\mu'(s) < 0$ for all $s > 0$,
\begin{equation}\label{322*}
\begin{aligned}
&\frac{1}{2}(||u_t||^2_2 + \mu(t)||~|\nabla_g u|_g||^2_2) + \frac{1}{\gamma + 2} ||u||^{\gamma+2}_{\gamma+2, \Gamma_1} + \int^t_0\int_{\Gamma_1} q(u_s(x,s)) u_s(x,s) d\Gamma ds
\\
&\leq \frac{1}{2} (||u_1||^2_2 + \mu(0)||~|\nabla_g u_0|_g||^2_2) + \frac{1}{\gamma + 2} ||u_0||^{\gamma+2}_{\gamma+2, \Gamma_1} + \int^t_0\int_\Omega f(x,s) u_s(x,s) dx ds
\\
& \hspace{5mm} + \Bigl(1 + \frac{1}{c_5} \Bigr) \int^t_0\int_{\Gamma_1} |h_K(u(x,s))|~ |u_s(x,s)| d\Gamma ds.
\end{aligned}
\end{equation}

We note that  $h_K : H^1_0(\Omega) \rightarrow L^{\frac{\rho+2}{\rho+1}}(\Gamma_1)$ is globally Lipschitz with Lipschitz constant $L_h(K)$ (see \cite{chue, guo3}). Hence we estimate the last term on the right-hand side of \eqref{322*} as follows:
\begin{equation}\label{323*}
\begin{aligned}
&\Bigl(1 + \frac{1}{c_5} \Bigr)\int^t_0\int_{\Gamma_1} |h_K(u(x,s))|~ |u_s(x,s)| d\Gamma ds
\\
& \leq \Bigl(1 + \frac{1}{c_5} \Bigr)\int^t_0 ||h_K(u(s))||_{\frac{\rho+2}{\rho+1}, \Gamma_1}||u_s||_{\rho+2, \Gamma_1} ds
\\
& \leq \epsilon_1 \int^t_0 ||u_s(s)||^{\rho+2}_{\rho+2, \Gamma_1} ds + C(\epsilon_1) \int^t_0 ||h_K(u(s))||^{\frac{\rho+2}{\rho+1}}_{\frac{\rho+2}{\rho+1}, \Gamma_1} ds
\\
& \leq \epsilon_1 \int^t_0 ||u_s(s)||^{\rho+2}_{\rho+2, \Gamma_1} ds + C(\epsilon_1)\Bigl(\frac{2}{\sqrt{c_1}}L_h(K)\Bigr)^{\frac{\rho+2}{\rho+1}} \int^t_0 ||~|\nabla_g u|_g||^2_2 ds
\\
& \hspace{5mm} + tC(\epsilon_1)\Bigl(\Bigl(\frac{2}{\sqrt{c_1}}L_h(K)\Bigr)^{\frac{\rho+2}{\rho+1}} + 2^{-(\rho+1)}|h(0)|^{\frac{\rho+2}{\rho+1}}meas(\Gamma_1)\Bigr).
\end{aligned}
\end{equation}

From the hypothesis on $q$, we have
\begin{equation}\label{324*}
\int^t_0\int_{\Gamma_1} q(u_s(x,s)) u_s(x,s) d\Gamma ds  \geq c_3 \int^t_0 ||u_s(s)||^{\rho+2}_{\rho+2, \Gamma_1} ds - tc_3 meas(\Gamma_1).
\end{equation}

By replacing \eqref{323*} and \eqref{324*} in \eqref{322*} and choosing $\epsilon_1 \leq c_3$, we get
\begin{multline}\label{325*}
||u_t||^2_2 + ||~|\nabla_g u|_g||^2_2 + ||u||^{\gamma+2}_{\gamma+2, \Gamma_1} +  \int^t_0 ||u_s(s)||^{\rho+2}_{\rho+2, \Gamma_1} ds
\\
\leq C_9 + C_1(L_h(K))T + C_2(L_h(K)) \int^t_0 ||u_s(s)||^2_2 + ||~|\nabla_g u(s)|_g||^2_2 ds
\end{multline}
for all $t \in [0, T]$, where
$$
C_9 = \frac{1}{2C_8}(||u_1||^2_2 + \mu(0)||~|\nabla_g u_0|_g||^2_2 + ||u_0||^{\gamma+2}_{\gamma+2, \Gamma_1} + ||f||_{L^2(0,T;L^2(\Omega))} ),
$$
$$
C_1(L_h(K)) = \frac{1}{C_8}C(\epsilon_1)\Bigl(\Bigl(\frac{2}{\sqrt{c_1}}L_h(K)\Bigr)^{\frac{\rho+2}{\rho+1}} + 2^{-(\rho+1)}|h(0)|^{\frac{\rho+2}{\rho+1}}meas(\Gamma_1)\Bigr) + \frac{1}{C_8}c_3 meas(\Gamma_1),
$$
$$
C_2(L_h(K)) = \frac{1}{C_8}\Bigl( C(\epsilon_1)\Bigl(\frac{2}{\sqrt{c_1}}L_h(K)\Bigr)^{\frac{\rho+2}{\rho+1}} + \frac{1}{2} \Bigr),
$$
for $C_8 = \min\{\frac{1}{\gamma+2}, \frac{\mu_0}{2}, c_3 -\epsilon_1\}$. Thus by Gronwall's inequality, \eqref{325*} becomes
$$
||u_t||^2_2 + ||~|\nabla_g u|_g||^2_2 + ||u||^{\gamma+2}_{\gamma+2,\Gamma_1} + \int^t_0 ||u_s(s)||^{\rho+2}_{\rho+2, \Gamma_1} ds \leq ( C_9 + C_1(L_h(K))T ) e^{C_2(L_h(K))t},
$$

for all $t \in [0,T]$. If we choose
\begin{equation}\label{3.30}
T = \min\Bigl\{\frac{1}{ C_1(L_h(K))}, \frac{1}{C_2(L_h(K))}\ln 2\Bigr\},
\end{equation}
then
\begin{equation}\label{326*}
||u_t||^2_2 + ||~|\nabla_g u|_g||^2_2 + ||u||^{\gamma+2}_{\gamma+2, \Gamma_1} + \int^t_0 ||u_s(s)||^{\rho+2}_{\rho+2, \Gamma_1} ds \leq 2 (C_9 + 1) \leq K^2 \hspace{3mm}\text{for all}\hspace{3mm} t \in [0, T],
\end{equation}
provided we choose $K^2 \geq 2(C_9 + 1)$. Consequently, \eqref{326*} gives us that $ ||~|\nabla_g u|_g||_2 \leq K$ for all $t \in [0,T]$. Therefore, by the definition of $h_K$, we have that $h_K(u) = h(u)$ on $[0,T]$. By the uniqueness of solutions, the solution of the truncated problem \eqref{321*} accords with the solution of the original, non-truncated problem \eqref{1} for $t \in [0,T]$, which means that the proof of Proposition 3.2 is completed.

\end{proof}

\subsection{Completion of the proof for the local existence}

In order to establish the existence of solutions, we need to extend the result in Proposition 3.2 where the source $h$ is locally Lipschitz from $H^1_0(\Omega)$ into $L^{\frac{\rho+2}{\rho+1}}(\Gamma_1)$. For the construction of the Lipschitz approximation for the source, we employ another truncated function introduced in \cite{radu}. Let $\delta_n \in C^{\infty}_0(\mathbb{R})$ be a cut off function such that
$$
\begin{cases}
0 \leq \delta_n \leq 1,
\\
\delta_n(s) = 1, \hspace{3mm}&\text{if}\hspace{3mm} |s| \leq n,
\\
\delta_n(s) = 0, \hspace{3mm}&\text{if}\hspace{3mm} |s| \geq 2n,
\end{cases}
$$
and $|\delta'_n(s)| \leq \frac{C}{n}$ for some constant $C$ independent from $n$ and define
\begin{equation}\label{327*}
h_n(u) = h(u) \delta_n(u).
\end{equation}
Then the truncated function $h_n$ is satisfied the following lemma. The proof of this lemma is a routine series of estimates as in \cite{boci1, guo3}, so we omit it here.
\begin{lem}
The following statements hold.
\begin{enumerate}
\item $h_n : H^1_0(\Omega) \rightarrow L^2(\Gamma_1)$ is globally Lipschitz continuous.
\\
\item $h_n : H^{1-\epsilon}_0(\Omega) \rightarrow L^{\frac{\rho+2}{\rho+1}}(\Gamma_1)$ is locally Lipschitz continuous with Lipschitz constant independent of $n$.
\end{enumerate}

\end{lem}

With the truncated source $h_n$ defined in \eqref{327*}, by Proposition 3.2 and Lemma 3.1, we have a unique local solution $u^n \in C(0,T ; \mathcal{H}) \cap C^1(0,T;L^2(\Omega))$ satisfying the following approximation of \eqref{1}
\begin{equation}\label{328*}
\begin{cases}
\vspace{3mm} u_{tt} - \mu(t) Lu = f(x,t) &\hspace{5mm} \text{in} \hspace{5mm}
\Omega~~\times~~ (0, +\infty),
\\
\vspace{3mm} u = 0 &\hspace{5mm} \text{on} \hspace{5mm} \Gamma_0 ~~\times~~ (0, +\infty),
\\
\vspace{3mm} \mu(t) \frac{\partial u}{\partial \nu_L} + q(u_t) = h_n(u) &\hspace{5mm} \text{on} \hspace{5mm} \Gamma_1 ~~\times~~ (0, +\infty),
\\
\vspace{3mm} u(x,0) = u_0(x), \hspace{5mm} u_t(x,0) = u_1(x).
\end{cases}
\end{equation}

From Lemma 3.1, the life span $T$ of each solution $u^n$, given in \eqref{3.30}, is independent of $n$ since the local Lipschitz constant of the mapping $h_n : H^1_0(\Omega) \rightarrow L^\frac{\rho+2}{\rho+1}(\Gamma_1)$ is independent of $n$. Also we known that $T$ depends on $K$, where $K^2 \geq 2(C_9 + 1)$, however, since $||u^n_1||^2_2 + ||~|\nabla_g u^n_0|_g||^2_2 + ||u^n_0||^{\gamma+2}_{\gamma+2, \Gamma_1} \rightarrow ||u_1||^2_2 + ||~|\nabla_g u_0|_g||^2_2 + ||u_0||^{\gamma+2}_{\gamma+2, \Gamma_1}$, we can choose $K$ sufficiently large so that $K$ is independent of $n$. By \eqref{326*},
\begin{equation}\label{329*}
||u^n_t||^2_2 + ||~|\nabla_g u^n|_g||^2_2 + ||u^n||^{\gamma+2}_{\gamma+2, \Gamma_1} \leq K^2
\end{equation}
for all $t \in [0, T]$. Therefore, there exists a function $u$ and a subsequence of $\{u^n\}$, which we still denote by $\{u^n\}$, such that
\begin{equation}\label{330*}
u^n \rightarrow u \hspace{3mm}\text{weak star in}\hspace{3mm} L^\infty (0, T ; \mathcal{H}),
\end{equation}
\begin{equation}\label{331*}
u^n_t \rightarrow u_t \hspace{3mm}\text{weak star in}\hspace{3mm} L^\infty (0, T ; L^2(\Omega)).
\end{equation}

By \eqref{329*}, \eqref{330*} and \eqref{331*}, we infer
\begin{equation}\label{332*}
||u_t||^2_2 + ||~|\nabla_g u|_g||^2_2 + ||u||^{\gamma+2}_{\gamma+2, \Gamma_1} \leq K^2
\end{equation}
for all $t \in [0, T]$. Moreover, by Aubin-Lions Theorem, we have
\begin{equation}\label{333*}
u^n \rightarrow u \hspace{3mm}\text{strongly in}\hspace{3mm} L^\infty (0, T ; H^{1-\epsilon}(\Omega)),
\end{equation}
for $0 < \epsilon < 1$. Since $u^n$ is a solution of \eqref{328*}, it holds that
\begin{multline}\label{334*}
\int^T_0 \int_\Omega u^n_{tt}, \phi dx dt + \int^T_0 \mu(t)\int_\Omega \langle \nabla_g u^n, \nabla_g \phi \rangle_g dx dt + \int^T_0\int_{\Gamma_1} q(u^n_t) \phi d\Gamma dt
\\
= \int^T_0 \int_\Omega f(x,t) \phi dx dt + \int^T_0\int_{\Gamma_1} h_n(u^n) \phi d\Gamma dt,
\end{multline}
for any $\phi \in C(0,T ; \mathcal{H}) \cap C^1 (0,T ; L^2(\Omega))$, $\phi  \in L^{\rho+2}(0, T ; \Gamma_1)$.

Now we will show that
\begin{equation}\label{335*}
\lim_{n\rightarrow \infty} \int^T_0\int_{\Gamma_1} h_n(u^n) \phi d\Gamma dt = \int^T_0\int_{\Gamma_1} h(u) \phi d\Gamma dt.
\end{equation}
Indeed, we have
\begin{equation}\label{336*}
\Bigl|  \int^T_0\int_{\Gamma_1} (h_n(u^n) - h(u)) \phi d\Gamma dt\Bigr| \leq  \int^T_0\int_{\Gamma_1} |h_n(u^n) - h_n(u)|~|\phi| d\Gamma dt + \int^T_0\int_{\Gamma_1} |h_n(u) - h(u)|~|\phi| d\Gamma dt.
\end{equation}

By $(2)$ in Lemma 3.1 and \eqref{333*}, we obtain
\begin{equation}\label{337*}
\begin{aligned}
\int^T_0\int_{\Gamma_1} |h_n(u^n) - h_n(u)|~|\phi| d\Gamma dt &\leq \Bigl( \int^T_0\int_{\Gamma_1} |h_n(u^n) - h_n(u)|^{\frac{\rho+2}{\rho+1}} d\Gamma dt \Bigr)^{\frac{\rho+1}{\rho+2}} \Bigl(\int^T_0\int_{\Gamma_1} |\phi|^{\rho+2} d\Gamma dt \Bigr)^{\frac{1}{\rho+2}}
\\
& \leq C(K) ||\phi||_{L^{\rho+2}(0, T ; \Gamma_1)} \Bigl( \int^T_0 ||u^n - u||^{\frac{\rho+2}{\rho+1}}_{H^{1-\epsilon}(\Omega)} dt \Bigr)^{\frac{\rho+1}{\rho+2}} \rightarrow 0.
\end{aligned}
\end{equation}

Since $\delta_n(u(x)) \rightarrow 1 $ a.e. in $\Omega$, we have $h_n(u) \rightarrow h(u)$ a.e. Then we also have $ |h_n(u) - h(u)|^{\frac{\rho+2}{\rho+1}} \leq 2^{\frac{\rho+2}{\rho+1}}|h(u)|^{\frac{\rho+2}{\rho+1}}$ and $h(u) \in L^{\frac{\rho+2}{\rho+1}}(\Gamma_1)$, for $u\in H^1_0(\Omega)$. Thus by the Lebesgue Dominated Convergence Theorem, we have
\begin{equation}\label{338*}
\begin{aligned}
\int^T_0\int_{\Gamma_1} |h_n(u) - h(u)|~|\phi| d\Gamma dt &\leq \Bigl( \int^T_0\int_{\Gamma_1} |h_n(u) - h(u)|^{\frac{\rho+2}{\rho+1}} d\Gamma dt \Bigr)^{\frac{\rho+1}{\rho+2}} \Bigl(\int^T_0\int_{\Gamma_1} |\phi|^{\rho+2} d\Gamma dt \Bigr)^{\frac{1}{\rho+2}}
\\
& \leq ||\phi||_{L^{\rho+2}(0, T ; \Gamma_1)} \Bigl( \int^T_0\int_{\Gamma_1} |h(u)|^{\frac{\rho+2}{\rho+1}} |\delta_n(u) - 1|^{\frac{\rho+2}{\rho+1}} d\Gamma dt \Bigr)^{\frac{\rho+1}{\rho+2}} \rightarrow 0.
\end{aligned}
\end{equation}

From convergences \eqref{337*} and \eqref{338*}, \eqref{336*} gives us \eqref{335*}.

On the other hand, by using similar arguments from \eqref{320} to \eqref{320*}, we get
\begin{equation}\label{340*}
q(u^n_t) \rightarrow q(u_t) \hspace{3mm}\text{weakly in}\hspace{3mm} L^2 (0, T ; L^2(\Gamma_1)).
\end{equation}

Convergences \eqref{331*}, \eqref{332*}, \eqref{335*} and \eqref{340*} permit us to pass to the limit in \eqref{334*} and conclude the following result.

\begin{prop}
Assume that $(H_1)-(H_4)$ hold. In addition, assume that  $(u_0,u_1) \in \mathcal{H} \times L^2(\Omega)$ and $h : H^1_0(\Omega) \rightarrow L^{\frac{\rho+2}{\rho+1}}(\Gamma_1)$ is locally Lipschitz continuous. Then problem \eqref{1} has a local solution $u \in C(0,T ; \mathcal{H}) \cap C^1 (0,T ; L^2(\Omega))$ for some $T > 0$.
\end{prop}

Let $h(u) = |u|^\gamma u$, then $h : H^1_0(\Omega) \rightarrow L^{\frac{\rho+2}{\rho+1}}(\Gamma_1)$ is locally Lipschitz continuous (see Remark 2.2). Thus by Proposition 3.3, the proof of the local existence statement in Theorem 2.1 is completed.

\subsection{Energy identity}

It is well known that to prove the uniqueness of weak solutions, we will justify the energy identity \eqref{213}. The energy identity can be derived formally by multiplying \eqref{1} by $u_t$. But, such a calculation is not justified, since $u_t$ is not sufficiently regular to be the test function in as required in Definition 2.1. To overcome this problem, we employ the operator $T^\epsilon = (I - \epsilon L)^{-1}$ to smooth function in space, which is mentioned in Appendix A of \cite{guo3}. We recall important properties of $T^\epsilon$ which play an essential role when establishing the energy identity.

\begin{lem} (\cite{guo3})
Let $u^\epsilon = T^\epsilon u$. Then following statements hold.
\begin{enumerate}
\item If $u \in L^2(\Omega)$, then $||u^\epsilon||_2 \leq ||u||_2$ and $u^\epsilon \rightarrow u$ in $L^2(\Omega)$ as $\epsilon \rightarrow 0$.
\\
\item If $u \in H^1_0(\Omega)$, then $||\nabla u^\epsilon||_2 \leq ||\nabla u||_2$ and $u^\epsilon \rightarrow u$ in $H^1_0(\Omega)$ as $\epsilon \rightarrow 0$.
\\
\item If $u \in L^p(\Gamma_1)$ with $1 < p < \infty$, then $||u^\epsilon||_{p,\Gamma_1} \leq ||u||_{p,\Gamma_1}$ and $u^\epsilon \rightarrow u$ in $L^p(\Gamma_1)$ as $\epsilon \rightarrow 0$.
\end{enumerate}
\end{lem}

We will now  justify the energy identity \eqref{213}. We play the operator $T^\epsilon$ on every term of \eqref{1} and multiply by $u^\epsilon_t$. Then we obtain by integrating in space and time
\begin{multline}\label{401*}
\int^t_0 \int_\Omega u^\epsilon_{ss} u^\epsilon_s dx ds + \int^t_0 \mu(s) \int_\Omega \langle \nabla_g u^\epsilon, \nabla_g u^\epsilon_s \rangle_g dx ds + \int^t_0\int_{\Gamma_1} T^\epsilon(q(u_s))u^\epsilon_s d\Gamma ds
\\
= \int^t_0 \int_\Omega f(x,s) u^\epsilon_s dx ds + \int^t_0\int_{\Gamma_1} T^\epsilon (h(u))u^\epsilon_s d\Gamma ds.
\end{multline}

Since $u \in H^1_0(\Omega)$ and $u_t \in L^2(\Omega)$, we have by Lemma 3.2, $u^\epsilon \rightarrow u$ in $H^1_0(\Omega)$ and $u^\epsilon_t \rightarrow u_t$ in $L^2(\Omega)$. Therefore using this convergences, we have
\begin{equation}\label{402*}
\lim_{\epsilon \rightarrow 0} \Bigl(\int^t_0 \int_\Omega u^\epsilon_{ss} u^\epsilon_s dx ds + \int^t_0 \mu(s) \int_\Omega \langle \nabla_g u^\epsilon, \nabla_g u^\epsilon_s \rangle_g dx ds \Bigr) = \frac{1}{2} \bigl(||u_t||^2_2 + ||~|\nabla_g u|_g||^2_2 - ||u_1||^2_2 -  ||~|\nabla_g u_0|_g||^2_2 \bigr).
\end{equation}

Since $u_t, q(u_t) \in L^2(\Gamma_1)$, we easily check that
\begin{equation}\label{403*}
\lim_{\epsilon \rightarrow 0} \int^t_0\int_{\Gamma_1} T^\epsilon(q(u_s))u^\epsilon_s d\Gamma ds =  \int^t_0\int_{\Gamma_1} (q(u_s))u_s d\Gamma ds.
\end{equation}

Recall that $u_t \in L^{\rho+2}(\Gamma_1)$ and $h(u) \in L^{\frac{\rho+2}{\rho+1}}(\Gamma_1)$. By Lemma 3.2, we have $u^\epsilon_t \rightarrow u_t$ in $L^{\rho+2}(\Gamma_1)$ and $T^\epsilon (h(u)) \rightarrow h(u)$ in $L^{\frac{\rho+2}{\rho+1}}(\Gamma_1)$. Thus by Lebesgue Dominated Convergence Theorem, we obtain
\begin{equation}\label{404*}
\lim_{\epsilon \rightarrow 0} \int^t_0\int_{\Gamma_1} T^\epsilon (h(u))u^\epsilon_s d\Gamma ds = \int^t_0\int_{\Gamma_1} (h(u))u_s d\Gamma ds.
\end{equation}
Convergences \eqref{402*}-\eqref{404*} permit us to pass to the limit in \eqref{401*}, consequently, the energy identity \eqref{213} holds.

\section{{\bf Proof of Theorem 2.1 : global existence}}
\setcounter{equation}{0}

In this section we prove that a local weak solution $u$ on $[0,T]$ can be extended to $[0,\infty)$. From the standard continuation argument of ODE theory, it suffices to show that $||u_t||^2_2 + ||~|\nabla_g u|_g ||^2_2 + ||u||^{\gamma+2}_{\gamma+2, \Gamma_1}$ is bounded independent of $t$. We now consider the following two cases:

\subsection{$\rho \geq \gamma$}

Using the energy identity \eqref{213}, we obtain
\begin{equation}\label{323}
\begin{aligned}
&\frac{d}{dt}\biggl[ \frac{1}{2} ||u_t||^2_2 + \frac{1}{2} \mu(t) ||~|\nabla_g u|_g ||^2_2 + \frac{1}{\gamma+2} ||u||^{\gamma+2}_{\gamma+2, \Gamma_1} \biggr]
\\
& = - \int_{\Gamma_1} q(u_t) u_t d\Gamma + \frac{1}{2} \mu'(t) ||~|\nabla_g u|_g ||^2_2 + 2 \int_{\Gamma_1} |u|^\gamma u u_t d\Gamma + \int_\Omega f(t) u_t dx.
\end{aligned}
\end{equation}

By the same argument as \eqref{303}, we have
\begin{equation}\label{324}
\int_{\Gamma_1} q(u_t)u_t d\Gamma \geq c_3 ||u_t||^{\rho+2}_{\rho+2, \Gamma_1} - c_3 meas(\Gamma_1).
\end{equation}
Using the H\"{o}lder and Young inequalities with $\frac{\gamma+1}{\gamma+2} + \frac{1}{\gamma+2} = 1$ and the imbedding $L^{\rho+2}(\Gamma_1) \hookrightarrow L^{\gamma+2}(\Gamma_1)$, we deduce that
\begin{equation}\label{325}
\begin{aligned}
2 \int_{\Gamma_1} |u|^\gamma u u_t d\Gamma & \leq C(\epsilon_2)||u||^{\gamma+2}_{\gamma+2, \Gamma_1} + \epsilon_2 C^{\gamma+2}_{\rho+2,\gamma+2}||u_t||^{\gamma+2}_{\rho+2, \Gamma_1}
\\
& \leq C(\epsilon_2)||u||^{\gamma+2}_{\gamma+2, \Gamma_1} + \epsilon_2 2^{\rho+1} C^{\gamma+2}_{\rho+2,\gamma+2}\bigl(1 + ||u_t||^{\rho+2}_{\rho+2, \Gamma_1} \bigr),
\end{aligned}
\end{equation}
where $C_{\rho+2,\gamma+2}$ is an imbedding constant. By replacing \eqref{324} and \eqref{325} in \eqref{323} and using the Young inequality and \eqref{208}, we get
\begin{equation}\label{326}
\begin{aligned}
&\frac{d}{dt}\biggl[ \frac{1}{2} ||u_t||^2_2 + \frac{1}{2} \mu(t) ||~|\nabla_g u|_g ||^2_2 + \frac{1}{\gamma+2} ||u||^{\gamma+2}_{\gamma+2, \Gamma_1} \biggr]
\\
&\leq \frac{1}{2} ||u_t||^2_2 + C(\epsilon_2) ||u||^{\gamma+2}_{\gamma+2, \Gamma_1} + (c_3 meas(\Gamma_1) + \epsilon_2 2^{\rho+1}C^{\gamma+2}_{\rho+2,\gamma+2} + \frac{1}{2}||f||^2_2)
\\
& \hspace{5mm} + (\epsilon_2 2^{\rho+1}C^{\gamma+2}_{\rho+2,\gamma+2} - c_3) ||u_t||^{\rho+2}_{\rho+2, \Gamma_1}.
\end{aligned}
\end{equation}

Let
$$
\widetilde{E}(t) = \frac{1}{2} ||u_t||^2_2 + \frac{1}{2} \mu(t) ||~|\nabla_g u|_g ||^2_2 + \frac{1}{\gamma+2} ||u||^{\gamma+2}_{\gamma+2, \Gamma_1}.
$$
Choosing $\epsilon_2 = \frac{c_3}{2^{\rho+1}C^{\gamma+2}_{\rho+2,\gamma+2}}$, we rewrite \eqref{326} as
$$
\widetilde{E}'(t) \leq C_{10} + C_{11} \widetilde{E}(t),
$$
where $C_{10}$ and $C_{11}$ are positive constants. Now applying Gronwall's inequality, we have that $\widetilde{E}(t) \leq (C_{12} \widetilde{E}(0) + C_{13})e^{C_{12} t}$, where $C_{12}$ and $C_{13}$ are positive constants. Consequently, since $\widetilde{E}(0)$ is bounded we conclude that $||u_t||^2_2 + ||~|\nabla_g u|_g ||^2_2 + ||u||^{\gamma+2}_{\gamma+2, \Gamma_1}$ is bounded.

\subsection{The potential well}

First of all, we will find a stable region. We set
$$
0 < K_0 := \sup_{u \in \mathcal{H}, u\neq 0} \Bigl( \frac{||u||_{\gamma + 2, \Gamma_1}}{||~|\nabla_g u|_g ||_2} \Bigr) < \infty
$$
and the functional
\begin{equation}\label{327}
J(u) = \frac{\mu_0}{2}||~|\nabla_g u|_g ||^2_2 - \frac{1}{\gamma+2}||u||^{\gamma+2}_{\gamma+2, \Gamma_1}, \hspace{5mm} u\in \mathcal{H}.
\end{equation}

We also define the function, for $\lambda > 0$,
\begin{equation}\label{328}
j(\lambda) = \frac{\mu_0}{2}\lambda^2 - \frac{1}{\gamma +2}K^{\gamma+2}_0 \lambda^{\gamma+2},
\end{equation}
then
$$
\lambda_0 = \Bigl(\frac{\mu_0}{K^{\gamma+2}_0} \Bigr)^{1/\gamma}
$$
is the absolute maximum point of $j$ and
$$
j(\lambda_0) = \frac{\gamma\mu_0}{2(\gamma+2)}\lambda^2_0 = d_0.
$$

The energy associated to the problem \eqref{1} is given by
\begin{equation}\label{329}
E(t) = \frac{1}{2}||u_t(t)||^2_2 + \frac{1}{2} \mu(t) ||~|\nabla_g u(t)|_g ||^2_2 - \frac{1}{\gamma+2} ||u(t)||^{\gamma+2}_{\gamma+2, \Gamma_1},
\end{equation}
for $u \in \mathcal{H}$. By \eqref{208} and \eqref{327}-\eqref{329}, we deduce
\begin{equation}\label{330}
E(t) \geq J(u(t)) \geq \frac{\mu_0}{2} ||~|\nabla_g u(t)|_g ||^2_2 - \frac{K^{\gamma+2}_0}{\gamma+2} ||~|\nabla_g u(t)|_g ||^{\gamma+2}_2 = j(||~|\nabla_g u(t)|_g ||_2).
\end{equation}

\begin{lem}

Let $u$ be a weak solution for problem \eqref{1}. Suppose that
$$
E(0) < d_0 \hspace{3mm}\text{and}\hspace{3mm} ||~|\nabla_g u_0|_g||_2 < \lambda_0.
$$
Then
$$
||~|\nabla_g u(t)|_g ||_2 < \lambda_0 \hspace{3mm}\text{for all}\hspace{3mm} t \geq 0.
$$

\end{lem}

\begin{proof}

It is easy to verify that $j$ is increasing for $0 < \lambda < \lambda_0$, decreasing for $\lambda > \lambda_0$, $j(\lambda) \rightarrow -\infty$ as $\lambda \rightarrow +\infty$. Then since $d_0 > E(0) \geq j(||~|\nabla_g u_0|_g||_2) \geq j(0) = 0$, there exist $\lambda'_0 < \lambda_0 < \tilde{\lambda_0}$, which verify
\begin{equation}\label{331}
j(\lambda'_0) = j(\tilde{\lambda_0}) = E(0).
\end{equation}
Considering that $E(t)$ is nonincreasing, we have
\begin{equation}\label{332}
E(t) \leq E(0) \hspace{3mm}\text{for all}\hspace{3mm} t\geq 0.
\end{equation}
From \eqref{330} and \eqref{331}, we deduce that
\begin{equation}\label{333}
j(||~|\nabla_g u_0|_g||_2) \leq E(0) = j(\lambda'_0).
\end{equation}
Since $||~|\nabla_g u_0|_g||_2 < \lambda_0$, $\lambda'_0 < \lambda_0$ and $j$ is increasing in $[0, \lambda_0)$, from \eqref{333} it holds that
\begin{equation}\label{334}
||~|\nabla_g u_0|_g||_2 \leq \lambda'_0.
\end{equation}

Next, we will prove that
\begin{equation}\label{335}
||~|\nabla_g u(t)|_g ||_2 \leq \lambda'_0 \hspace{3mm}\text{for all}\hspace{3mm} t \geq 0.
\end{equation}
We argue by contradiction. Suppose that \eqref{335} does not hold. Then there exists time $t^*$ which verifies
\begin{equation}\label{336}
||~|\nabla_g u(t^*)|_g ||_2 > \lambda'_0.
\end{equation}

If $||~|\nabla_g u(t^*)|_g ||_2 < \lambda_0$, from \eqref{330}, \eqref{331} and \eqref{336} we can write
$$
E(t^*) \geq j(||~|\nabla_g u(t^*)|_g ||_2) > j(\lambda'_0) = E(0),
$$
which contradicts \eqref{332}.

If $||~|\nabla_g u(t^*)|_g ||_2 \geq \lambda_0$, then we have, in view of \eqref{334}, that there exists $\bar{\lambda_0}$ which verifies
\begin{equation}\label{337}
||~|\nabla_g u_0|_g ||_2 \leq \lambda'_0 < \bar{\lambda_0} < \lambda_0 \leq ||~|\nabla_g u(t^*)|_g ||_2.
\end{equation}
Consequently, from the continuity of the function $||~|\nabla_g u(\cdot)|_g ||_2$ there exists $\bar{t} \in (0,t^*)$ verifying
\begin{equation}\label{338}
||~|\nabla_g u(\bar{t})|_g ||_2 = \bar{\lambda_0}.
\end{equation}
Then from \eqref{330}, \eqref{331}, \eqref{337} and \eqref{338}, we get
$$
E(\bar{t}) \geq j(||~|\nabla_g u(\bar{t})|_g ||_2) = j(\bar{\lambda_0}) > j(\lambda'_0) = E(0),
$$
which also contradicts \eqref{332}. This completes the proof of Lemma 4.1.

\end{proof}

From \eqref{330} and Lemma 4.1, we arrive at
\begin{equation}\label{339}
E(t) \geq J(u(t)) >||~|\nabla_g u(t)|_g ||^2_2 \Bigl( \frac{\mu_0}{2} - \frac{K^{\gamma+2}_0}{\gamma+2}\lambda^\gamma_0 \Bigr) = \mu_0||~|\nabla_g u(t)|_g ||^2_2 \Bigl( \frac{1}{2} - \frac{1}{\gamma+2} \Bigr)
\end{equation}
and, consequently,
\begin{equation}\label{340}
J(t) \geq 0 ~~(J(t) = 0 ~~\text{iff}~~ u = 0) \hspace{3mm} \text{and} \hspace{3mm} ||~|\nabla_g u(t)|_g ||^2_2 \leq \frac{2(\gamma+2)}{\mu_0\gamma} E(t).
\end{equation}

By virtue of \eqref{339}, we get
\begin{equation}\label{341}
J(u(t)) > \frac{\mu_0 \gamma}{2(\gamma+2)}||~|\nabla_g u(t)|_g ||^2_2.
\end{equation}
Hence
$$
\frac{1}{2}||u_t(t)||^2_2 + \frac{\mu_0 \gamma}{2(\gamma+2)}||~|\nabla_g u(t)|_g ||^2_2 < \frac{1}{2}||u_t(t)||^2_2 + J(u(t)) \leq E(t) \leq E(0).
$$
Therefore, there exists a positive constant $C_{14}$ independent of $t$ such that
\begin{equation}\label{342}
||u_t(t)||^2_2 + ||~|\nabla_g u(t)|_g ||^2_2 \leq C_{14} E(0).
\end{equation}
Moreover, if we define the functional $I(u(t))$ by
$$
I(u(t)) = \mu_0 ||~|\nabla_g u(t)|_g ||^2_2 - ||u(t)||^{\gamma+2}_{\gamma+2, \Gamma_1},
$$
then from the relationship $I(u(t)) = (\gamma+2)J(u(t)) - \frac{\mu_0\gamma}{2} ||~|\nabla_g u(t)|_g ||^2_2$ and the strict inequality \eqref{341}, we obtain
\begin{equation}\label{343}
I(u(t)) > 0 \hspace{3mm}\text{for all}\hspace{3mm} t \geq 0.
\end{equation}
Consequently, from \eqref{342} and \eqref{343} we have
$$
||u_t(t)||^2_2 + ||~|\nabla_g u(t)|_g ||^2_2 + ||u(t)||^{\gamma+2}_{\gamma+2, \Gamma_1} \leq (1 + \mu_0) C_{14} E(0).
$$
This it the completion of the proof of the global existence of solutions of \eqref{1}.

\section{ {\bf Proof of Theorem 2.2 : energy decay} }
\setcounter{equation}{0}

In this section we prove the uniform decay rates for the solution of the following problem:
\begin{equation}\label{401}
\begin{cases}
\vspace{3mm} u_{tt} - \mu(t) Lu = 0 &\hspace{5mm} \text{in} \hspace{5mm}
\Omega~~\times~~ (0, +\infty),
\\
\vspace{3mm} u = 0 &\hspace{5mm} \text{on} \hspace{5mm} \Gamma_0 ~~\times~~ (0, +\infty),
\\
\vspace{3mm} \mu(t) \frac{\partial u}{\partial \nu_L} + q(u_t) = |u|^\gamma u
&\hspace{5mm} \text{on} \hspace{5mm} \Gamma_1 ~~\times~~ (0, +\infty),
\\
\vspace{3mm} u(x,0) = u_0(x), \hspace{5mm} u_t(x,0) = u_1(x),
\end{cases}
\end{equation}

We consider the following additional hypothesis on $H$:
\begin{equation}\label{402}
\sigma \leq div(H) \leq \frac{\sigma(\gamma+4)}{\gamma+2}.
\end{equation}
Unless otherwise stated, the constant $C$ is a generic positive constant, different in various occurrences. We define the energy associated to problem \eqref{401}:
$$
E(t) = \frac{1}{2} ||u_t||^2_2 + \frac{1}{2} \mu(t) ||~|\nabla_g u|_g ||^2_2 - \frac{1}{\gamma+2} ||u||^{\gamma+2}_{\gamma+2, \Gamma_1}.
$$
Then
$$
E'(t) = \frac{1}{2} \mu'(t)||~|\nabla_g u|_g ||^2_2 -\int_{\Gamma_1} q(u_t)u_t d\Gamma \leq 0,
$$
it follows that $E(t)$ is a nonincreasing function.

First of all, we recall technical lemmas which will play an essential role when establishing the asymptotic behavior.

\begin{lem}(\cite{mart})
Let $E : \mathbb{R}_+ \rightarrow \mathbb{R}_+$ be a nonincreasing function and $\phi :
\mathbb{R}_+ \rightarrow \mathbb{R}_+$ a strictly increasing function of class $C^1$ such
that
\begin{equation*}
\phi(0) = 0 \hspace{5mm}\text{and} \hspace{5mm} \phi(t)\rightarrow +\infty
\hspace{5mm}\text{as}\hspace{5mm} t \rightarrow +\infty.
\end{equation*}
Assume that there exists $\sigma \geq 0$ and $\omega>0$ such that
\begin{equation*}
\int^{+\infty}_S E^{1+\sigma}(t) \phi'(t) dt \leq \frac{1}{\omega} E^\sigma (0) E(S)
\end{equation*}
for all $S \geq 0$. Then $E$ has the following decay property:
\begin{equation*}
\text{if}\hspace{5mm} \sigma = 0, \hspace{5mm}\text{then}\hspace{5mm} E(t) \leq E(0) e^{1-\omega\phi(t)}, \hspace{5mm}\text{for all}\hspace{5mm} t \geq 0,
\end{equation*}
\begin{equation*}
\text{if}\hspace{5mm} \sigma > 0, \hspace{5mm}\text{then}\hspace{5mm} E(t) \leq E(0) \Bigl(\frac{1+\sigma}{1+\omega\sigma\phi(t)}\Bigr)^\frac{1}{\sigma}, \hspace{5mm}\text{for all}\hspace{5mm} t \geq 0.
\end{equation*}
\end{lem}

\begin{lem}(\cite{mart})
Let $E : \mathbb{R}_+ \rightarrow \mathbb{R}_+$ be a nonincreasing function and $\phi : \mathbb{R}_+ \rightarrow \mathbb{R}_+$ a strictly increasing function of class $C^1$ such that
\begin{equation*}
\phi(0) = 0 \hspace{5mm}\text{and} \hspace{5mm} \phi(t)\rightarrow +\infty \hspace{5mm}\text{as}\hspace{5mm} t \rightarrow +\infty.
\end{equation*}
Assume that there exists $\sigma > 0$, $\sigma' \geq 0$ and $C>0$ such that
\begin{equation*}
\int^{+\infty}_S E^{1+\sigma}(t) \phi'(t) dt \leq CE^{1+\sigma}(S) + \frac{C}{(1+ \phi(S))^{\sigma'}} E^\sigma(0) E(S), \hspace{5mm} 0 \leq S < +\infty.
\end{equation*}
Then, there exists $C>0$ such that
\begin{equation*}
E(t) \leq E(0) \frac{C}{(1 + \phi(t))^{(1+\sigma')/\sigma}}, \hspace{5mm} \forall t >0.
\end{equation*}
\end{lem}

Let us now multiply equation \eqref{401} by $E^p(t)\phi'(t) \mathcal{M}u$, where $\mathcal{M}u$ is given by
$$
\mathcal{M}u = 2H(u) + (div (H) - \sigma) u,
$$
$ p\geq 0$ and $\phi : \mathbb{R} \rightarrow \mathbb{R}$ is a concave nondecreasing function of class $C^2$, such that $\phi(t)\rightarrow +\infty$ as $ t \rightarrow +\infty$, and then integrate the obtained result over $\Omega \times [S, T]$. Then we have
\begin{equation}\label{403}
\begin{aligned}
0 &= \int^T_S E^p(t) \phi'(t) \int_\Omega \mathcal{M}u \Bigl(u_{tt} - \mu(t) L u\Bigr) dx dt
\\
& = \int^T_S E^p(t) \phi'(t) \int_\Omega u_{tt} \mathcal{M}u dx dt - \int^T_S E^p(t) \phi'(t) \int_\Omega  (div (H) - \sigma) u \mu(t)Lu dx dt
\\
&\hspace{5mm} - 2\int^T_S E^p(t) \phi'(t) \int_\Omega H(u) \mu(t)Lu dx dt.
\end{aligned}
\end{equation}

We note that
\begin{align*}
&\int^T_S E^p(t) \phi'(t) \int_\Omega u_{tt} \mathcal{M}u dx dt
\\
& = \Bigl[E^p(t)\phi'(t) \int_\Omega u_t \mathcal{M}u dx \Bigr]^T_S - \int^T_S (pE^{p-1}(t)E'(t)\phi'(t) + E^p(t)\phi''(t))\int_\Omega u_t \mathcal{M}u dx dt
\\
& \hspace{5mm} -2\int^T_S E^p(t) \phi'(t) \int_\Omega u_t H(u_t) dx dt - \int^T_S E^p(t) \phi'(t) \int_\Omega  (div (H) - \sigma) |u_t|^2 dx dt
\end{align*}
\begin{align*}
&- \int^T_S E^p(t) \phi'(t) \int_\Omega  (div (H) - \sigma) u \mu(t)Lu dx dt
\\
&= - \int^T_S E^p(t) \phi'(t) \mu(t) \int_{\Gamma_1} (div(H)- \sigma) u \frac{\partial u}{\partial \nu_L} d\Gamma dt + \int^T_S E^p(t) \phi'(t) \mu(t)\int_\Omega (div(H) - \sigma) |\nabla_g u|^2_g dx dt
\end{align*}
and using Lemma 2.1 and the fact $H(u)\frac{\partial u}{\partial \nu_L} = |\nabla_g u|^2_g$ on $\Gamma_0$,
\begin{align*}
& - 2 \int^T_S E^p(t) \phi'(t) \int_\Omega H(u) \mu(t)Lu dx dt
\\
&= - \int^T_S E^p(t) \phi'(t) \mu(t) \int_{\Gamma_1} \biggl( 2\frac{\partial u}{\partial \nu_L} H(u) - |\nabla_g u|^2_g (H\cdot \nu) \biggr) d\Gamma dt
\\
& \hspace{5mm} - \int^T_S E^p(t) \phi'(t) \mu(t) \int_{\Gamma_0} |\nabla_g u|^2_g (H\cdot \nu) d\Gamma dt  + 2 \int^T_S E^p(t) \phi'(t) \mu(t) \int_\Omega D_gH(\nabla_g u, \nabla_g u) dx dt
\\
& \hspace{5mm} - \int^T_S E^p(t) \phi'(t) \mu(t) \int_\Omega | \nabla_g u|^2_g div (H) dx dt.
\end{align*}

By replacing above identities in \eqref{403}, we obtain
\begin{equation}\label{404}
\begin{aligned}
&\sigma \int^T_S E^p(t) \phi'(t) \int_\Omega |u_t|^2 dxdt + 2 \int^T_S E^p(t) \phi'(t) \mu(t) \int_\Omega D_gH(\nabla_g u, \nabla_g u) dx dt
\\
&\hspace{5mm} - \sigma \int^T_S E^p(t) \phi'(t) \mu(t) \int_\Omega | \nabla_g u|^2_g dx dt - (div(H) - \sigma) \int^T_S E^p(t) \phi'(t) \int_{\Gamma_1} |u|^{\gamma+2} d\Gamma dt
\\
& = - \Bigl[E^p(t)\phi'(t) \int_\Omega u_t \mathcal{M}u dx \Bigr]^T_S + \int^T_S (pE^{p-1}(t)E'(t)\phi'(t) + E^p(t)\phi''(t))\int_\Omega u_t \mathcal{M}u dx dt
\\
&\hspace{5mm} + 2 \int^T_S E^p(t) \phi'(t) \int_{\Gamma_1} |u|^\gamma u H(u) d\Gamma dt
\\
&\hspace{5mm} + \int^T_S E^p(t) \phi'(t) \int_{\Gamma_1} q(u_t) \mathcal{M}u + \bigl(|u_t|^2 - \mu(t) | \nabla_g u|^2_g \bigr) (H\cdot \nu) d\Gamma dt
\\
& := I_1 + I_2 + I_3 + I_4.
\end{aligned}
\end{equation}

Now we are going to estimate terms on the right hand side of \eqref{404}.

$Estimate ~~for~~I_1 := -\Bigl[E^p(t)\phi'(t) \int_\Omega u_t \mathcal{M}u dx \Bigr]^T_S$ ;

Using the Young inequality and the inequality
$$
\int_\Omega |u|^2 dx \leq c^*_\Omega \int_\Omega |\nabla_g u|^2_g dx, \hspace{3mm} c^*_\Omega > 0, \hspace{3mm} \forall u \in H^1_0(\Omega),
$$
we obtain
\begin{equation}\label{405}
\Bigl| \int_\Omega u_t \mathcal{M}u dx \Bigr| \leq CE(t),
\end{equation}
consequently,
\begin{equation}\label{406}
I_1 \leq -C \Bigl[ E^p(t) \phi'(t) E(t)\Bigr]^T_S \leq CE^{p+1}(S).
\end{equation}

$Estimate ~~for~~I_2 := \int^T_S \Bigl(pE^{p-1}(t) E'(t)\phi'(t) + E^p(t)\phi''(t) \Bigr) \int_\Omega u_t \mathcal{M}u dx dt$ ;

From \eqref{405}, we have
\begin{multline}\label{407}
|I_2| \leq C \int^T_S |pE^{p-1}(t)E'(t)\phi'(t) + E^p(t)\phi''(t)| E(t) dt
\\
\leq CE^p(S) \int^T_S -E'(t) dt + CE^{p+1}(S) \int^T_S -\phi''(t) dt \leq CE^{p+1}(S).
\end{multline}

$Estimate ~~for~~I_3 := 2 \int^T_S E^p(t) \phi'(t) \int_{\Gamma_1} |u|^\gamma u H(u) d\Gamma dt$ ;

By the Young inequality with $\frac{\rho+1}{\rho+2} + \frac{1}{\rho+2} = 1$ and from the fact $k(u) = |u|^\gamma u$ is locally Lipschitz from $H^1_0(\Omega)$ into $L^\frac{\rho+2}{\rho+1}(\Gamma_1)$, \eqref{201} and \eqref{340} we get
\begin{equation}\label{408}
\begin{aligned}
\int_{\Gamma_1} |u|^\gamma u H(u) d\Gamma & \leq C(\epsilon_3) \int_{\Gamma_1} |u|^\frac{(\gamma+1)(\rho+2)}{\rho+1} d\Gamma + \epsilon_3 \int_{\Gamma_1} |H(u)|^{\rho+2} d\Gamma
\\
& \leq C(\epsilon_3) L^\frac{\rho+2}{\rho+1}_\gamma ||\nabla u||^\frac{\rho+2}{\rho+1}_2 + \epsilon_3 \sup_{x \in \overline{\Omega}} |H|^{\rho+2}_g \int_{\Gamma_1} |\nabla_g u|^{\rho +2}_g d\Gamma
\\
& \leq C(\epsilon_3) L^\frac{\rho+2}{\rho+1}_\gamma c^{-1}_1 ||~|\nabla_g u|_g||^2_2 + \epsilon_3 \sup_{x \in \overline{\Omega}} |H|^{\rho+2}_g \int_{\Gamma_1} |\nabla_g u|^{\rho +2}_g d\Gamma
\\
& \leq C(\epsilon_3) L^\frac{\rho+2}{\rho+1}_\gamma \frac{2(\gamma+2)}{\mu_0 \gamma c_1} E(t) + \epsilon_3 \sup_{x \in \overline{\Omega}} |H|^{\rho+2}_g \int_{\Gamma_1} |\nabla_g u|^{\rho +2}_g d\Gamma,
\end{aligned}
\end{equation}
consequently,
\begin{equation}\label{409}
I_3 \leq C(\epsilon_3) E^{p+1}(S) +  \epsilon_3 \sup_{x \in \overline{\Omega}} |H|^{\rho+2}_g \int^T_S E^p(t) \phi'(t) \int_{\Gamma_1} |\nabla_g u|^{\rho +2}_g d\Gamma dt.
\end{equation}

$Estimate ~~for~~I_4 :=  \int^T_S E^p(t) \phi'(t) \int_{\Gamma_1} q(u_t) \mathcal{M}u + \bigl(|u_t|^2 - \mu(t) | \nabla_g u|^2_g \bigr) (H\cdot \nu) d\Gamma dt$ ;

From the Young inequality with $\frac{\rho+1}{\rho+2} + \frac{1}{\rho+2} = 1$, we have
\begin{align*}
2\int_{\Gamma_1} q(u_t) H(u) d\Gamma &\leq C(\epsilon_4) \int_{\Gamma_1} |q(u_t)|^\frac{\rho+2}{\rho+1} d\Gamma + \epsilon_4 \int_{\Gamma_1} |H(u)|^{\rho+2} d\Gamma
\\
& \leq  C(\epsilon_4) \int_{\Gamma_1} |q(u_t)|^\frac{\rho+2}{\rho+1} d\Gamma + \epsilon_4 \sup_{x \in \overline{\Omega}} |H|^{\rho+2}_g \int_{\Gamma_1} |\nabla_g u|^{\rho +2}_g d\Gamma.
\end{align*}
Similar arguments as \eqref{408} we have
$$
(div(H) - \sigma)\int_{\Gamma_1} q(u_t) u d\Gamma \leq C \int_{\Gamma_1} |q(u_t)|^\frac{\alpha}{\alpha - 1} d\Gamma + C \int_{\Gamma_1}  |u|^\alpha d\Gamma \leq C \int_{\Gamma_1} |q(u_t)|^\frac{\alpha}{\alpha - 1} d\Gamma + CE(t),
$$
where $\alpha = \frac{(\gamma+1)(\rho+2)}{\rho+1}$. Hence we obtain
\begin{equation}\label{410}
\begin{aligned}
I_4 &\leq CE^{p+1}(S) + C(\epsilon_4) \int^T_S E^p(t) \phi'(t) \int_{\Gamma_1} |q(u_t)|^\frac{\rho+2}{\rho+1} d\Gamma dt + C \int^T_S E^p(t) \phi'(t) \int_{\Gamma_1} |q(u_t)|^\frac{\alpha}{\alpha - 1} d\Gamma dt
\\
&\hspace{5mm} + C \int^T_S E^p(t) \phi'(t) \int_{\Gamma_1} |u_t|^2 d\Gamma dt + \epsilon_4 \sup_{x \in \overline{\Omega}} |H|^{\rho+2}_g \int^T_S E^p(t) \phi'(t) \int_{\Gamma_1} |\nabla_g u|^{\rho +2}_g d\Gamma dt
\\
& \hspace{5mm} - \delta \mu_0 \int^T_S E^p(t) \phi'(t) \int_{\Gamma_1} |\nabla_g u|^2_g d\Gamma dt.
\end{aligned}
\end{equation}

By replacing \eqref{406}, \eqref{407}, \eqref{409} and \eqref{410} in \eqref{404} and choosing $\epsilon_3$, $\epsilon_4$ small enough, we obtain from \eqref{402}
\begin{equation}\label{411}
\begin{aligned}
\int^T_S E^{p+1}(t) \phi'(t) dt &\leq CE^{p+1}(S) + C\underbrace{\int^T_S E^p(t) \phi'(t) \int_{\Gamma_1} |u_t|^2 d\Gamma dt}_{:=I_5}
\\
& \hspace{5mm} + C \underbrace{\int^T_S E^p(t) \phi'(t) \int_{\Gamma_1} |q(u_t)|^\frac{\rho+2}{\rho+1} d\Gamma dt}_{:=I_6} + C \underbrace{\int^T_S E^p(t) \phi'(t) \int_{\Gamma_1} |q(u_t)|^\frac{\alpha}{\alpha - 1} d\Gamma dt}_{:=I_7}.
\end{aligned}
\end{equation}

Now we are going to estimate the last three terms on the right hand side of \eqref{411}.

\subsection{ \textit{Case 1 : $\beta$ is linear.}}

Since $\beta$ is linear, we can rewrite the hypothesis of $q$ as follows:
$$
c_7|s| \leq |q(s)| \leq c_8|s| \hspace{3mm} \text{if} \hspace{3mm} |s| \leq 1,
$$
$$
c_3 |s| \leq c_3 |s|^{\rho+1} \leq |q(s)| \leq c_4 |s|^{\rho+1} \hspace{3mm} \text{if} \hspace{3mm} |s| > 1,
$$
for some positive constants $c_7$, $c_8$. Hence we get
\begin{equation}\label{412}
\int_{|u_t| \leq 1} |u_t|^2 d\Gamma \leq c^{-1}_7 \int_{|u_t| \leq 1} u_t q(u_t) d\Gamma \leq -c^{-1}_7 E'(t),
\end{equation}
\begin{equation}\label{413}
\int_{|u_t| > 1} |u_t|^2 d\Gamma \leq c^{-1}_3 \int_{|u_t| > 1} u_t q(u_t) d\Gamma \leq -c^{-1}_3 E'(t),
\end{equation}
\begin{equation}\label{414}
\int_{|u_t| \leq 1} |q(u_t)|^\frac{\rho+2}{\rho+1} d\Gamma \leq \int_{|u_t| \leq 1} |q(u_t)|^2 d\Gamma \leq  c_8 \int_{|u_t| \leq 1} u_t q(u_t) d\Gamma \leq -c_8 E'(t),
\end{equation}
\begin{equation}\label{415}
\int_{|u_t| > 1} |q(u_t)|^\frac{\rho+2}{\rho+1} d\Gamma = \int_{|u_t| > 1}  |q(u_t)|^\frac{1}{\rho+1} |q(u_t)| d\Gamma \leq c^\frac{1}{\rho+1}_4 \int_{|u_t| > 1} u_t q(u_t) d\Gamma \leq -c^\frac{1}{\rho+1}_4 E'(t),
\end{equation}
\begin{equation}\label{416}
\int_{|u_t| \leq 1} |q(u_t)|^\frac{\alpha}{\alpha-1} d\Gamma \leq c_8 \int_{|u_t| \leq 1} u_t q(u_t) d\Gamma \leq -c_8 E'(t)
\end{equation}
and, since $\rho \leq \gamma$, it holds that $\frac{\rho+1}{\alpha - 1} \leq 1$. Consequently,
\begin{equation}\label{417}
\int_{|u_t| > 1} |q(u_t)|^\frac{\alpha}{\alpha-1} d\Gamma = \int_{|u_t| > 1}  |q(u_t)|^\frac{1}{\alpha - 1} |q(u_t)| d\Gamma \leq c^\frac{1}{\alpha - 1}_4 \int_{|u_t| > 1} u_t q(u_t) d\Gamma \leq -c^\frac{1}{\alpha - 1}_4 E'(t).
\end{equation}

From \eqref{412}-\eqref{417}, we obtain
\begin{equation}\label{418}
I_5 + I_6 + I_7 \leq C E^{p+1}(S).
\end{equation}
Combining \eqref{411} and \eqref{418}, it follows that
$$
\int^T_S E^{p+1}(t) \phi'(t) dt \leq C E^p(0) E(S),
$$
which implies by Lemma 5.1 with $p=0$
$$
E(t)  \leq E(0) e^{1 - \frac{\phi(t)}{C}}.
$$

Let us set $\phi(t) := mt$, where $m$ is for some positive constant, then $\phi(t)$ satisfies all the required properties and we obtain that the energy decays exponentially to zero.

\subsection{ \textit{Case 2 : $\beta$ has polynomial growth near zero.}}

Assume that $\beta(s) = s^{\rho+1}$. Let $p = \frac{\rho}{2}$, then we rewrite \eqref{411} as
\begin{equation}\label{419}
\begin{aligned}
\int^T_S E^{\frac{\rho}{2}+1}(t) \phi'(t) dt &\leq CE(S) + C\int^T_S E^\frac{\rho}{2}(t) \phi'(t) \int_{\Gamma_1} |u_t|^2 d\Gamma dt + C \int^T_S E^\frac{\rho}{2}(t) \phi'(t) \int_{\Gamma_1} |q(u_t)|^\frac{\rho+2}{\rho+1} d\Gamma dt
\\
& \hspace{5mm} + C \int^T_S E^\frac{\rho}{2}(t) \phi'(t) \int_{\Gamma_1} |q(u_t)|^\frac{\alpha}{\alpha - 1} d\Gamma dt.
\end{aligned}
\end{equation}

By hypotheses on $q$ and the H\"{o}lder inequality with $\frac{2}{\rho+2} + \frac{\rho}{\rho+2} = 1$, we have
$$
\int_{|u_t|\leq 1}  |u_t|^2 d\Gamma \leq \int_{|u_t|\leq 1} (u_tq(u_t))^\frac{2}{\rho+2} d\Gamma \leq C \biggl(\int_{|u_t|\leq  1} u_tq(u_t)d\Gamma \biggr)^\frac{2}{\rho+2} \leq C \bigl(-E'(t)\bigr)^\frac{2}{\rho+2}
$$
and
$$
\int_{|u_t|> 1} |u_t|^2 d\Gamma \leq c^{2(\rho+2)}_3\int_{|u_t| > 1} (u_tq(u_t))^\frac{2}{\rho+2} d\Gamma \leq C \bigl(-E'(t)\bigr)^\frac{2}{\rho+2}.
$$
Hence
\begin{equation}\label{420}
\begin{aligned}
&\int^T_S E^\frac{\rho}{2}(t)\phi'(t) \int_{\Gamma_1} |u_t|^2 d\Gamma dt
\\
& = \int^T_S E^\frac{\rho}{2}(t)\phi'(t) \int_{|u_t|\leq 1}|u_t|^2 d\Gamma dt + \int^T_S E^\frac{\rho}{2}(t)\phi'(t) \int_{|u_t|>1} |u_t|^2 d\Gamma dt
\\
& \leq C \int^T_S E^\frac{\rho}{2}(t)\phi'(t) \bigl(-E'(t)\bigr)^\frac{2}{\rho+2} dt
\\
& \leq \epsilon_5\int^T_S E^{\frac{\rho}{2}+1}(t)\phi'(t) dt + C(\epsilon_5) E(S).
\end{aligned}
\end{equation}

Similarly as \eqref{414} and \eqref{415} we have
$$
\int_{|u_t| \leq 1} |q(u_t)|^\frac{\rho+2}{\rho+1} d\Gamma \leq \int_{|u_t| \leq 1} |q(u_t)|^2 d\Gamma \leq \int_{|u_t|\leq 1} (u_tq(u_t))^\frac{2}{\rho+2} d\Gamma \leq C \bigl(-E'(t)\bigr)^\frac{2}{\rho+2}
$$
and
$$
\int_{|u_t| > 1} |q(u_t)|^\frac{\rho+2}{\rho+1} d\Gamma \leq -c^\frac{1}{\rho+1}_4 E'(t).
$$
Hence
\begin{equation}\label{421}
\int^T_S E^\frac{\rho}{2}(t) \phi'(t) \int_{\Gamma_1} |q(u_t)|^\frac{\rho+2}{\rho+1} d\Gamma dt \leq \epsilon_6\int^T_S E^{\frac{\rho}{2}+1}(t)\phi'(t) dt + C(\epsilon_6) E(S).
\end{equation}

Similarly as \eqref{416} and \eqref{417} we have
$$
\int_{|u_t| \leq 1} |q(u_t)|^\frac{\alpha}{\alpha-1} d\Gamma  \leq \int_{|u_t| \leq 1} |q(u_t)|^2 d\Gamma \leq C \bigl(-E'(t)\bigr)^\frac{2}{\rho+2}
$$
and
$$
\int_{|u_t| > 1} |q(u_t)|^\frac{\alpha}{\alpha-1} d\Gamma \leq -c^\frac{1}{\alpha - 1}_4 E'(t).
$$
Hence
\begin{equation}\label{422}
\int^T_S E^\frac{\rho}{2}(t) \phi'(t) \int_{\Gamma_1} |q(u_t)|^\frac{\alpha}{\alpha - 1} d\Gamma dt \leq \epsilon_7\int^T_S E^{\frac{\rho}{2}+1}(t)\phi'(t) dt + C(\epsilon_7) E(S).
\end{equation}

By replacing \eqref{420}-\eqref{422} in \eqref{419} and choosing $\epsilon_5, \epsilon_6, \epsilon_7$ sufficiently small, we get
$$
\int^T_S E^{\frac{\rho}{2}+1}(t) \phi'(t) dt \leq C E(S),
$$
which implies by Lemma 5.1 and choosing $\phi(t) = mt$,
$$
E(t) \leq \frac{CE(0)}{(1+t)^\frac{2}{\rho}}.
$$

\subsection{ \textit{Case 3 : $\beta$ does not necessarily have polynomial growth near zero.}}

We will use the method of partitions of boundary modified the arguments in \cite{mart}. For every $t \geq 1$, we consider the following partitions of boundary depending $\phi'(t)$:
$$
\Gamma_1^1 = \{ x \in \Gamma_1 ; |u_t(t)| \leq \phi'(t) \}, \hspace{3mm} \Gamma_1^2 = \{ x \in \Gamma_1 ; \phi'(t) < |u_t(t)| \leq 1 \},\hspace{3mm} \Gamma_1^3 = \{ x \in \Gamma_1 ; |u_t(t)| > 1 \}
$$
if $\phi'(t) \leq 1$, or
$$
\Gamma_1^4 = \{ x \in \Gamma_1 ; |u_t(t)| \leq 1 < \phi'(t) \}, \Gamma_1^5 = \{ x \in \Gamma_1 ; 1 < |u_t(t)| \leq \phi'(t) \}, \Gamma_1^6 = \{ x \in \Gamma_1 ; |u_t(t)| > \phi'(t) > 1 \}
$$
if $\phi'(t) > 1$. Then $\Gamma_1 = \Gamma_1^1 \cup \Gamma_1^2 \cup \Gamma_1^3$ (or $\Gamma_1 = \Gamma_1^4 \cup \Gamma_1^5 \cup \Gamma_1^6$). Let us estimate $I_5$, $I_6$ and $I_7$ on these partitions.

\textbf{(I)} Part on $\Gamma_1^i$, $i = 3, 5, 6$.

By same arguments as \eqref{413}, \eqref{415} and \eqref{417} we get
\begin{equation}\label{423}
\int^T_S E^p(t) \phi'(t)\int_{\Gamma_1^i} |u_t|^2 d\Gamma dt \leq C E^{p+1}(S),
\end{equation}
\begin{equation}\label{424}
\int^T_S E^p(t) \phi'(t)\int_{\Gamma_1^i} |q(u_t)|^\frac{\rho+2}{\rho+1} d\Gamma dt \leq C E^{p+1}(S)
\end{equation}
and
\begin{equation}\label{425}
\int^T_S E^p(t) \phi'(t)\int_{\Gamma_1^i} |q(u_t)|^\frac{\alpha}{\alpha-1} d\Gamma dt \leq C E^{p+1}(S).
\end{equation}

\textbf{(II)} Part on $\Gamma_1^2$.

Using the fact $\beta$ is increasing, $\phi'$ is nonincreasing and \eqref{210}, we obtain
\begin{equation}\label{426}
\begin{aligned}
\int^T_S E^p(t) \phi'(t)\int_{\Gamma_1^2} |u_t|^2 d\Gamma dt &\leq \int^T_S \frac{E^p(t)}{\beta(\phi'(t))} \int_{\Gamma_1^2} u_t q(u_t) d\Gamma dt
\\
&\leq \frac{1}{\beta(\phi'(T))}\int^T_S E^p(t)\bigl(-E'(t)\bigr) dt
\\
&\leq  \frac{1}{\beta(\phi'(T))}E^{p+1}(S),
\end{aligned}
\end{equation}
\begin{equation}\label{427}
\begin{aligned}
\int^T_S E^p(t) \phi'(t)\int_{\Gamma_1^2} |q(u_t)|^\frac{\rho+2}{\rho+1} d\Gamma dt &\leq \int^T_S E^p(t)\int_{\Gamma_1^2} |u'(t)|~|q(u_t)|^2 d\Gamma dt
\\
&\leq \beta^{-1}(1) \int^T_S E^p(t)\int_{\Gamma_1^2} u_t q(u_t) d\Gamma dt
\\
&\leq \beta^{-1}(1) E^{p+1}(S)
\end{aligned}
\end{equation}
and
\begin{equation}\label{428}
\int^T_S E^p(t) \phi'(t)\int_{\Gamma_1^2} |q(u_t)|^\frac{\alpha}{\alpha-1} d\Gamma dt \leq \int^T_S E^p(t)\int_{\Gamma_1^2} |u'(t)|~|q(u_t)|^2 d\Gamma dt \leq \beta^{-1}(1) E^{p+1}(S).
\end{equation}

\textbf{(III)} Part on $\Gamma_1^i$, $i = 1, 4$.

Using the fact $E(t)$ is nonincreasing and $\beta^{-1}$ is increasing, we have
\begin{equation}\label{429}
\begin{aligned}
\int^T_S E^p(t) \phi'(t)\int_{\Gamma_1^i} |u_t|^2 d\Gamma dt &\leq \frac{E^p(S)}{(\beta^{-1}(\phi'(T)))^2} \int^T_S \int_{\Gamma_1^i}\phi'(t) (\beta^{-1}(\phi'(t)))^2 d\Gamma dt
\\
& \leq \frac{meas(\Gamma_1)}{(\beta^{-1}(\phi'(T)))^2}E^p(S) \int^T_S \phi'(t) (\beta^{-1}(\phi'(t)))^2 dt.
\end{aligned}
\end{equation}
From \eqref{210}, we obtain
\begin{equation}\label{430}
\begin{aligned}
\int^T_S E^p(t) \phi'(t)\int_{\Gamma_1^i} |q(u_t)|^\frac{\rho+2}{\rho+1} d\Gamma dt &\leq \int^T_S E^p(t) \phi'(t)\int_{\Gamma_1^i} |q(u_t)|^2 d\Gamma dt
\\
&\leq  \int^T_S E^p(t) \phi'(t)\int_{\Gamma_1^i} (\beta^{-1}(|u'(t)|))^2 d\Gamma dt
\\
&\leq meas(\Gamma_1) E^p(S) \int^T_S \phi'(t) (\beta^{-1}(\phi'(t)))^2 dt
\end{aligned}
\end{equation}
and
\begin{equation}\label{431}
\int^T_S E^p(t) \phi'(t)\int_{\Gamma_1^i} |q(u_t)|^\frac{\alpha}{\alpha-1} d\Gamma dt \leq meas(\Gamma_1) E^p(S) \int^T_S \phi'(t) (\beta^{-1}(\phi'(t)))^2 dt.
\end{equation}
Therefore from \eqref{423}-\eqref{431}, we deduce that
\begin{equation}\label{432}
I_5 + I_6 + I_7 \leq CE^{p+1}(S) +  CE^p(S) \int^T_S \phi'(t) (\beta^{-1}(\phi'(t)))^2 dt.
\end{equation}

To estimate the last term of the right hand side of \eqref{432}, we need the following additional assumption over $\phi$ (see \cite{mart}, p.434):
$$
\int^\infty_1 \phi'(t) (\beta^{-1}(\phi'(t)))^2 dt \hspace{3mm} \text{converges}.
$$
Then by replacing \eqref{432} in \eqref{411} we obtain
\begin{equation}\label{433}
\begin{aligned}
\int^T_S E^{p+1}(t) \phi'(t) dt & \leq C E^{p+1}(S) + CE^p(S) \int^{+\infty}_S \phi'(t)(\beta^{-1}(\phi'(t)))^2 dt
\\
& \leq C E^{p+1}(S) + CE^p(S) \int^{+\infty}_{\phi(S)} \Bigl( \beta^{-1}(\frac{1}{(\phi^{-1})'(s)})\Bigr)^2 ds.
\end{aligned}
\end{equation}

Define $\psi(t) = 1 + \int^t_1 \frac{1}{\beta(\frac{1}{s})} ds$, $t \geq 1$. Then $\psi$ is strictly increasing and convex (cf. \cite{mart}, \cite{park4}). We now take $\phi(t) = \psi^{-1}(t)$, then we can rewrite \eqref{433} as
$$
\int^T_S E^{p+1}(t) \phi'(t) dt \leq C E^{p+1}(S) + \frac{C}{\phi(S)}E^p(S),
$$
which implies, by applying Lemma 5.2 with $p = 1$,
$$
E(t) \leq \frac{C}{\phi^2(t)} \hspace{5mm}\forall t > 0.
$$

Let $s_0$ be a number such that $ \beta(\frac{1}{s_0}) \leq 1$. Since $\beta$ is nondecreasing, we have
$$
\psi(s) \leq 1 + (s -1) \frac{1}{\beta(\frac{1}{s})} \leq \frac{1}{F(\frac{1}{s})} \hspace{5mm}\forall s \geq s_0,
$$
where $F(s) = s\beta(s)$, consequently, having in mind that $\phi = \psi^{-1}$, the last inequality yields
$$
s \leq \phi\Bigl( \frac{1}{F(\frac{1}{s})}\Bigr) = \phi(t) \hspace{5mm}\text{with}\hspace{5mm} t = \frac{1}{F(\frac{1}{s})}.
$$
Then we conclude that
$$
\frac{1}{\phi(t)} \leq F^{-1}(\frac{1}{t}).
$$
Therefore the proof of Theorem 2.2 is completed.

\section{ {\bf Proof of Theorem 2.3 : blow-up} }
\setcounter{equation}{0}

This section is devoted to prove the blow-up result. First of all, we introduce a following lemma that is essential role for proving the blow-up.
\begin{lem}
Under the hypotheses given in Theorem 2.3 the weak solution to problem \eqref{1} verifies
$$
||~|\nabla_g u(t)|_g ||_2 > \lambda_0 \hspace{3mm}\text{for all}\hspace{3mm} 0 < t < T_{\max}.
$$
\end{lem}

\begin{proof}

We recall the function, for $\lambda > 0$,
$$
j(\lambda) = \frac{\mu_0}{2}\lambda^2 - \frac{1}{\gamma +2}K^{\gamma+2}_0 \lambda^{\gamma+2},
$$
where $ K_0 = \sup_{u \in \mathcal{H}, u \neq 0} \Bigl(\frac{||u||_{\gamma+2, \Gamma_1}}{||~|\nabla_g u|_g ||_2} \Bigr)$. Then
$$
\lambda_0 = \Bigl(\frac{\mu_0}{K^{\gamma+2}_0} \Bigr)^{1/\gamma}
$$
is the absolute maximum point of $j$ and
$$
j(\lambda_0) = \frac{\gamma\mu_0}{2(\gamma+2)}\lambda^2_0 = d_0.
$$

The energy associated to problem \eqref{1} is given by
$$
E(t) = \frac{1}{2}||u_t(t)||^2_2 + \frac{1}{2} \mu(t) ||~|\nabla_g u(t)|_g ||^2_2 - \frac{1}{\gamma+2} ||u(t)||^{\gamma+2}_{\gamma+2, \Gamma_1}.
$$
We observe that from the definition of $j$, we have
\begin{equation}\label{501}
E(t) \geq \frac{\mu_0}{2} ||~|\nabla_g u(t)|_g ||^2_2 - \frac{1}{\gamma + 2} ||u(t)||^{\gamma+2}_{\gamma+2, \Gamma_1} \geq j(||~|\nabla_g u(t)|_g ||_2) \hspace{3mm} \text{for all} \hspace{3mm} t \geq 0.
\end{equation}

Note that $j$ is increasing for $0 <\lambda <\lambda_0$, decreasing for $\lambda > \lambda_0$, $j(\lambda) \rightarrow -\infty$ as $\lambda \rightarrow+\infty$.

We will now consider the initial energy $E(0)$ divided into two cases: $E(0) \geq 0$ and $E(0) < 0$.

\textbf{Case 1} : $E(0) \geq 0$.

There exist $\lambda'_1 < \lambda_0 < \lambda_1$ such that
\begin{equation}\label{502}
j(\lambda_1) = j(\lambda'_1) = E(0).
\end{equation}

By considering that $E(t)$ is nonincreasing, we have
\begin{equation}\label{503}
E(t) \leq E(0) \hspace{3mm}\text{for all} \hspace{3mm} t>0.
\end{equation}
From \eqref{501} and \eqref{502} we deduce
\begin{equation}\label{504}
j(||~|\nabla_g u_0|_g ||_2) \leq E(0) = j(\lambda_1).
\end{equation}
Since $||~|\nabla_g u_0|_g ||_2 > \lambda_0$, $\lambda_0 < \lambda_1$ and $j(\lambda)$ is decreasing for $\lambda_0 < \lambda$, from \eqref{504} we get
\begin{equation}\label{505}
||~|\nabla_g u_0|_g ||_2 \geq \lambda_1.
\end{equation}

Now we will prove that
\begin{equation}\label{506}
||~|\nabla_g u(t)|_g ||_2 \geq \lambda_1 \hspace{3mm}\text{for all}\hspace{3mm} 0 < t < T_{\max}
\end{equation}
by using the contradiction method. Suppose that \eqref{506} does not hold. Then there exists $t^* \in (0, T_{\max})$ which verifies
\begin{equation}\label{507}
||~|\nabla_g u(t^*)|_g ||_2 < \lambda_1.
\end{equation}

If $||~|\nabla_g u(t^*)|_g ||_2 > \lambda_0$, from \eqref{501}, \eqref{502} and \eqref{507} we can write
$$
E(t^*) \geq j(||~|\nabla_g u(t^*)|_g ||_2) > j(\lambda_1) = E(0),
$$
which contradicts \eqref{503}.

If $||~|\nabla_g u(t^*)|_g ||_2 \leq \lambda_0$, we have, in view of \eqref{505}, that there exists $\bar{\lambda}$ which verifies
\begin{equation}\label{508}
||~|\nabla_g u(t^*)|_g ||_2 \leq \lambda_0 < \bar{\lambda} < \lambda_1 \leq ||~|\nabla_g u_0|_g ||_2.
\end{equation}
Consequently, from the continuity of the function $||~|\nabla_g u(\cdot)|_g ||_2$ there exists $\bar{t} \in (0,t^*)$ verifying $||~|\nabla_g u(\bar{t})|_g ||_2 = \bar{\lambda}$. Then from the last identity and
taking \eqref{501}, \eqref{502} and \eqref{508} into account we deduce
$$
E(\bar{t}) \geq j(||~|\nabla_g u(\bar{t})|_g ||_2) = j(\bar{\lambda}) > j(\lambda_1) = E(0),
$$
which also contradicts \eqref{503}.

\textbf{Case 2} : $E(0) < 0$.

There is $\lambda_2 > \lambda_0$ such that
$$
j(\lambda_2) = E(0),
$$
consequently, by \eqref{501} we have
$$
j(||~|\nabla_g u_0|_g ||_2) \leq E(0) = j(\lambda_2).
$$
From the fact $j(\lambda)$ is decreasing for $\lambda_0 < \lambda$, we get
$$
||~|\nabla_g u_0|_g ||_2 \geq \lambda_2.
$$
By the same argument as Case 1, we obtain
$$
||~|\nabla_g u(t)|_g ||_2 \geq \lambda_2 \hspace{3mm}\text{for all}\hspace{3mm} 0 < t < T_{\max}.
$$

Thus the proof of Lemma 6.1 is completed.

\end{proof}

Now we will prove the blow-up result. In order to prove that $T_{\max}$ is necessarily finite, we argue by contradiction. Assume that the weak solution $u(t)$ can be extended to the whole interval $[0, \infty)$.

Let $E_1$ be a real number such that
$$
E_1 =
\begin{cases}
0 \hspace{3mm}&\text{if}\hspace{3mm} E(0) < 0,
\\
\text{positive constant satisfying}\hspace{3mm} E(0) < E_1 < d_0 \hspace{3mm}\text{and}\hspace{3mm} E_1 < E(0) + 1 \hspace{3mm}&\text{if}\hspace{3mm} E(0) \geq 0.
\end{cases}
$$
By setting $G(t) := E_1 - E(t)$, we have
\begin{equation}\label{509}
G'(t) = - E'(t) \geq 0,
\end{equation}
which implies that $G(t)$ is nondecreasing, consequently,
\begin{equation}\label{510}
0<  G_0 := E_1 - E(0) <1
\end{equation}
and from Lemma 6.1, \eqref{208} and the definition of $d_0$,
\begin{equation}\label{511}
\begin{aligned}
G_0 \leq G(t) &\leq E_1 - \frac{\mu_0}{2}||~|\nabla_g u(t)|_g ||^2_2 + \frac{1}{\gamma+2}||u(t)||^{\gamma+2}_{\gamma+2, \Gamma_1}
\\
&< d_0 - \frac{\mu_0}{2}\lambda^2_0 + \frac{1}{\gamma+2}||u(t)||^{\gamma+2}_{\gamma+2, \Gamma_1}
\\
&\leq  \frac{1}{\gamma+2}||u(t)||^{\gamma+2}_{\gamma+2, \Gamma_1}.
\end{aligned}
\end{equation}

We define
\begin{equation}\label{512}
M(t)= G^{1-\overline{\chi}}(t) + \tau N(t), \hspace{3mm} N(t) = \int_\Omega u_t u dx,
\end{equation}
where $\overline{\chi}$ and $\tau$ are small positive constants to be chosen later. Then we have
\begin{equation}\label{513}
M'(t) = (1-\overline{\chi})G^{-\overline{\chi}}(t)G'(t) + \tau N'(t).
\end{equation}

We are now going to analyze the last term on the right-hand side of \eqref{513}.
\begin{lem}
\begin{equation}\label{514}
\begin{aligned}
N'(t) &\geq  C_{15} \bigl( ||u_t||^2_2 + ||u||^{\gamma+2}_{\gamma+2, \Gamma_1} + G(t) - G'(t) G^{\overline{\chi}- \chi}_0 G^{-\overline{\chi}}(t) \bigr)
\\
&\hspace{5mm} + \mu_0\Bigl(\frac{\theta}{2} -1\Bigr) ||~|\nabla_g u|_g ||^2_2 - \theta E_1 - \frac{\zeta}{\epsilon_8},
\end{aligned}
\end{equation}
where $C_{15}$ is for some positive constant, $0 < \chi < \frac{\gamma-\rho}{(\rho+2)(\gamma+2)}$, $\theta = \gamma +2 - \epsilon_8$ with $0 < \epsilon_8 < \min\{1, \gamma\}$ and $\zeta =  \frac{(\gamma+1)meas(\Gamma_1) \bigl(\beta^{-1}(1)\bigr)^\frac{\gamma+2}{\gamma+1}}{\gamma+2}$.
\end{lem}

\begin{proof}

Using Eq. \eqref{1}, we obtain
\begin{equation}\label{515}
\begin{aligned}
N'(t) &= ||u_t||^2_2 - \mu(t) ||~|\nabla_g u|_g ||^2_2 + ||u||^{\gamma+2}_{\gamma+2, \Gamma_1} - \int_{\Gamma_1} q(u_t) u d\Gamma
\\
&\geq \Bigl(1+\frac{\theta}{2} \Bigr)||u_t||^2_2 + \mu_0\Bigl(\frac{\theta}{2} -1\Bigr) ||~|\nabla_g u|_g ||^2_2  + \Bigl(1 - \frac{\theta}{\gamma+2}\Bigr) ||u||^{\gamma+2}_{\gamma+2, \Gamma_1} + \theta G(t) - \theta E_1 \\
& \hspace{5mm} - \int_{\Gamma_1} q(u_t) u d\Gamma,
\end{aligned}
\end{equation}
where $\theta = \gamma +2 - \epsilon_8$ with $0 < \epsilon_8 < \min\{1, \gamma\}$.

We will now estimate the last term on the right-hand side of \eqref{515}. We note that
$$
\Bigl|\int_{\Gamma_1} q(u_t) u d\Gamma\Bigr| \leq \int_{\Gamma_1} |q(u_t)| ~|u| d\Gamma = \int_{|u_t| \leq 1} |q(u_t)| ~|u| d\Gamma + \int_{|u_t| > 1} |q(u_t)| ~|u| d\Gamma.
$$

By using \eqref{210} and the imbedding $L^{\gamma+2}(\Gamma_1)\hookrightarrow L^{\rho+2}(\Gamma_1)$, we have
\begin{equation}\label{516}
\begin{aligned}
\int_{|u_t| \leq 1} |q(u_t)| ~|u| d\Gamma &\leq \Bigl(\int_{|u_t| \leq 1} |q(u_t)|^\frac{\rho+2}{\rho+1} d\Gamma \Bigr)^\frac{\rho+1}{\rho+2} ~\Bigl(\int_{|u_t| \leq 1} |u|^{\rho+2} d\Gamma \Bigr)^\frac{1}{\rho+2}
\\
& \leq \Bigl(\int_{|u_t| \leq 1} |\beta^{-1}(1)|^\frac{\rho+2}{\rho+1} d\Gamma \Bigr)^\frac{\rho+1}{\rho+2} ~||u||_{\rho+2, \Gamma_1}
\\
& \leq \beta^{-1}(1) \bigl(meas(\Gamma_1)\bigr)^\frac{\gamma+1}{\gamma+2} ~||u|||_{\gamma+2, \Gamma_1}
\\
& \leq \frac{(\gamma+1)meas(\Gamma_1) \bigl(\beta^{-1}(1)\bigr)^\frac{\gamma+2}{\gamma+1}}{\epsilon_8 (\gamma+2)} + \frac{\epsilon^{\gamma+1}_8}{\gamma+2} ||u||^{\gamma+2}_{\gamma+2, \Gamma_1}.
\end{aligned}
\end{equation}

On the other hand, by using \eqref{211}, we obtain
\begin{equation}\label{517}
\begin{aligned}
\int_{|u_t| > 1} |q(u_t)| ~|u| d\Gamma &\leq c_4\int_{|u_t| > 1}|u_t|^{\rho+1} |u| d\Gamma
\\
& \leq c_4 \Bigl(\int_{|u_t| > 1}|u_t|^{\rho+2} d\Gamma\Bigr)^\frac{\rho+1}{\rho+2} ||u||_{\rho+2, \Gamma_1}
\\
& \leq \Bigl( C(\epsilon_9)\int_{|u_t| > 1}|u_t|^{\rho+2} d\Gamma + \epsilon_9||u||^{(\gamma+2)(\chi+\frac{1}{\gamma+2})(\rho+2)}_{\gamma+2, \Gamma_1} \Bigr)~ ||u||^{-(\gamma+2)\chi}_{\gamma+2, \Gamma_1},
\end{aligned}
\end{equation}
where $0 < \chi < \frac{\gamma-\rho}{(\rho+2)(\gamma+2)}$ and  $C(\epsilon_9)$, $\epsilon_9$ are for some positive constants. Moreover $\chi < \frac{\gamma-\rho}{(\rho+2)(\gamma+2)}$ implies that $(\chi+\frac{1}{\gamma+2})(\rho+2)< 1$. Hence we get
$$
||u||^{(\gamma+2)(\chi+\frac{1}{\gamma+2})(\rho+2)}_{\gamma+2, \Gamma_1} \leq
\begin{cases}
||u||^{\gamma+2}_{\gamma+2, \Gamma_1} &\hspace{3mm}\text{if}\hspace{3mm} ||u||^{\gamma+2}_{\gamma+2, \Gamma_1}>1,
\\
G^{-1}_0G_0 &\hspace{3mm}\text{if}\hspace{3mm} ||u||^{\gamma+2}_{\gamma+2, \Gamma_1} \leq 1.
\end{cases}
$$
From \eqref{510} and \eqref{511} we have
$$
||u||^{(\gamma+2)(\chi+\frac{1}{\gamma+2})(\rho+2)}_{\gamma+2, \Gamma_1} \leq G^{-1}_0 ||u||^{\gamma+2}_{\gamma+2, \Gamma_1}
$$
and, consequently, from \eqref{211}, \eqref{509}, \eqref{511} and \eqref{517},
\begin{equation}\label{518}
\begin{aligned}
\int_{|u_t| > 1} |q(u_t)| ~|u| d\Gamma &\leq  \Bigl( C(\epsilon_9)\int_{|u_t| > 1}|u_t|^{\rho+2} d\Gamma + \epsilon_9 G^{-1}_0 ||u||^{\gamma+2}_{\gamma+2, \Gamma_1} \Bigr)~ ||u||^{-(\gamma+2)\chi}_{\gamma+2, \Gamma_1}
\\
& \leq \Bigl( C(\epsilon_9) G'(t) + \epsilon_9 G^{-1}_0 ||u||^{\gamma+2}_{\gamma+2, \Gamma_1} \Bigr)~ G^{-\chi}(t)
\\
&\leq  C(\epsilon_9) G'(t) G^{\overline{\chi}- \chi}_0 G^{-\overline{\chi}}(t) + \epsilon_9 G^{-(\chi+1)}_0 ||u||^{\gamma+2}_{\gamma+2, \Gamma_1},
\end{aligned}
\end{equation}
for $0< \overline{\chi} < \chi$. From \eqref{516} and \eqref{518}, we get that
\begin{equation}\label{519}
\Bigl|\int_{\Gamma_1} q(u_t) u d\Gamma \Bigr| \leq  C(\epsilon_9) G'(t) G^{\overline{\chi}- \chi}_0 G^{-\overline{\chi}}(t) + \Bigl( \frac{\epsilon^{\gamma+1}_8}{\gamma+2} + \epsilon_9 G^{-(\chi+1)}_0\Bigr) ||u||^{\gamma+2}_{\gamma+2, \Gamma_1} + \frac{\zeta}{\epsilon_8},
\end{equation}
where $\zeta =  \frac{(\gamma+1)meas(\Gamma_1) \bigl(\beta^{-1}(1)\bigr)^\frac{\gamma+2}{\gamma+1}}{\gamma+2}$.

By replacing \eqref{519} in \eqref{515} and choosing $\epsilon_9$ small enough we obtain
$$
N'(t) \geq  C_{16} \bigl( ||u_t||^2_2 + ||u||^{\gamma+2}_{\gamma+2, \Gamma_1} + G(t) - G'(t) G^{\overline{\chi}- \chi}_0 G^{-\overline{\chi}}(t) \bigr) +  \mu_0\Bigl(\frac{\theta}{2} -1\Bigr) ||~|\nabla_g u|_g ||^2_2 - \theta E_1 - \frac{\zeta}{\epsilon_8},
$$
where $C_{16}$ is a positive constant. Therefore \eqref{514} follows.

\end{proof}

The following Lemma estimates the last three terms on the right-hand side of \eqref{514}.

\begin{lem}
\begin{equation}\label{520}
\mu_0\Bigl(\frac{\theta}{2} -1\Bigr) ||~|\nabla_g u|_g ||^2_2 - \theta E_1 - \frac{\zeta}{\epsilon_8} > 0  \hspace{3mm}\text{for}\hspace{3mm} \frac{\ell - \sqrt{\ell^2 - 4\varrho\zeta}}{2\varrho} \leq \epsilon_8 \leq \frac{\ell + \sqrt{\ell^2 - 4\varrho\zeta}}{2\varrho},
\end{equation}
where $\varrho = \frac{\mu_0\lambda^2_0}{2} - E_1$ and $\ell = \frac{\gamma\mu_0\lambda^2_0}{2} - (\gamma+2) E_1$.

\end{lem}

\begin{proof}

From Lemma 5.1 and the definition of $\theta$, we have
\begin{multline}\label{521}
\mu_0\Bigl(\frac{\theta}{2} -1\Bigr) ||~|\nabla_g u|_g ||^2_2 - \theta E_1 - \frac{\zeta}{\epsilon_8} > \mu_0\Bigl(\frac{\theta}{2} -1\Bigr)\lambda^2_0 - \theta E_1 - \frac{\zeta}{\epsilon_8}
\\
= \Bigl(E_1 - \frac{\mu_0\lambda^2_0}{2}\Bigr) \epsilon_8 - \frac{\zeta}{\epsilon_8} + \frac{\gamma\mu_0\lambda^2_0}{2} - (\gamma+2) E_1  = - \frac{\varrho\epsilon^2_8 - \ell\epsilon_8 + \zeta}{\epsilon_8}:= P(\epsilon_8).
\end{multline}

We note that
$$
\varrho = \frac{\mu_0\lambda^2_0}{2} - E_1 > \frac{\mu_0\lambda^2_0}{2} - d_0 = \frac{1}{\gamma + 2} K^{\gamma+2}_0 \lambda^{\gamma+2}_0 > 0
$$
and
$$
\ell = \frac{\gamma\mu_0\lambda^2_0}{2} - (\gamma+2) E_1 > \frac{\gamma\mu_0\lambda^2_0}{2} - (\gamma+2) d_0 = 0.
$$

Since \eqref{215} holds, we get
$$
\ell^2 - 4\varrho \zeta \geq 0.
$$
Therefore, $P(\epsilon_8)$ represents a curve connecting horizontal axis points $\frac{\ell - \sqrt{\ell^2 - 4\varrho\zeta}}{2\varrho}$ and $\frac{\ell + \sqrt{\ell^2 - 4\varrho\zeta}}{2\varrho}$, and
$$
P(\epsilon_8) \geq 0 \hspace{3mm}\text{for}\hspace{3mm} \frac{\ell - \sqrt{\ell^2 - 4\varrho\zeta}}{2\varrho} \leq \epsilon_8 \leq \frac{\ell + \sqrt{\ell^2 - 4\varrho\zeta}}{2\varrho}.
$$

\begin{figure}[ht]
\begin{center}
\resizebox{0.40\textwidth}{!}{%
  \includegraphics{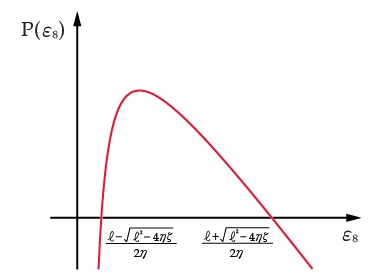}}
\caption{The figure of $P(\epsilon_8)$}\label{}
\end{center}
\end{figure}

Thus we obtain
$$
\mu_0\Bigl(\frac{\theta}{2} -1\Bigr) ||~|\nabla_g u|_g ||^2_2 - \theta E_1 - \frac{\zeta}{\epsilon_8} > 0  \hspace{3mm}\text{for}\hspace{3mm} \frac{\ell - \sqrt{\ell^2 - 4\varrho\zeta}}{2\varrho} \leq \epsilon_8 \leq \frac{\ell + \sqrt{\ell^2 - 4\varrho\zeta}}{2\varrho}.
$$

\end{proof}

Combining \eqref{513}, \eqref{514}, \eqref{520} and then choosing $0 < \overline{\chi} < \min\{\frac{1}{2}, \chi\}$ and $\tau$ small enough, we obtain
$$
M'(t) \geq C_{17} \bigl( ||u_t||^2_2 + ||u||^{\gamma+2}_{\gamma+2, \Gamma_1} + G(t) \bigr),
$$
where $C_{17}$ is a positive constant, which implies that $M(t)$ is a positive increasing function. By same arguments as p.333 in \cite{ha3}, we have
$$
M'(t) \geq C_{18} M^\frac{1}{1-\overline{\chi}}(t) \hspace{3mm}\text{for all}\hspace{3mm} t\geq 0,
$$
where $C_{18}$ is a positive constant and $1 < \frac{1}{1-\overline{\chi}} < 2$. Hence we conclude that $M(t)$ blows up in finite time and $u$ also blows up in finite time. Thus this is a contradiction, consequently, the proof of Theorem 2.3 is completed.

\section*{Acknowledgments}

This research was supported by Basic Science Research Program through the National Research Foundation of Korea(NRF) funded by the Ministry of Education (2022R1I1A3055309).



\end{document}